\documentclass[12pt]{article}
\usepackage{amssymb}
\usepackage{graphicx}
\usepackage[all]{xy}
\usepackage{color}
\usepackage{latexsym}
\usepackage{amssymb,amsmath,amstext}
\usepackage{epsfig}
\usepackage{graphics}
\usepackage{euscript}
\usepackage{ifthen}
\usepackage{wrapfig}
\usepackage{amsthm}
\usepackage{multicol}
\usepackage{paralist}
\usepackage{verbatim}
\usepackage[colorlinks=true, urlcolor=blue, linkcolor=blue, citecolor=blue]{hyperref}
\usepackage[margin=1in]{geometry}
\usepackage{algorithmicx}
\usepackage{algpseudocode}
\usepackage{algorithm}
\usepackage{placeins}
\usepackage{caption}
\usepackage{subcaption}
\usepackage{longtable}
\usepackage{lscape}
\usepackage{placeins}
\oddsidemargin \evensidemargin
\usepackage{listings}

\makeatletter
\newtheorem*{rep@theorem}{\rep@title}
\newcommand{\newreptheorem}[2]{%
\newenvironment{rep#1}[1]{%
 \def\rep@title{#2 \ref{##1}}%
 \begin{rep@theorem}}%
 {\end{rep@theorem}}}
\makeatother

\numberwithin{equation}{section}
\theoremstyle{plain}%
\newtheorem{theorem}{Theorem}
\numberwithin{theorem}{section}
\newtheorem{proposition}[theorem]{Proposition}
\newtheorem{example}[theorem]{Example}
\newtheorem{lemma}[theorem]{Lemma}
\newtheorem{corollary}[theorem]{Corollary}
\newtheorem{definition}[theorem]{Definition}
\newtheorem{remark}[theorem]{Remark}
\newtheorem{conjecture}[theorem]{Conjecture}

\newreptheorem{theorem}{Conjecture}

\newcommand{\R}{\mathbb{R}}

\newcommand{\sM}{\mathcal{M}}

\newcommand{\cone}{\mathrm{cone}}
\newcommand{\conv}{\mathrm{conv}}
\newcommand{\rank}{\mathrm{rank}}

\allowdisplaybreaks

\date{\today}
\begin{document}
\title{\textbf{Positive semidefinite rank and nested spectrahedra}}
\author{Kaie Kubjas, Elina Robeva and Richard Z. Robinson}
\maketitle

\begin{abstract} The set of matrices of given positive semidefinite rank is semialgebraic. In this paper we study the geometry of this set, and in small cases we describe its boundary. For general values of positive semidefinite rank we provide a conjecture for the description of this boundary. Our proof techniques are geometric in nature and rely on nesting spectrahedra between polytopes.
\end{abstract}

\section{Introduction}

Standard matrix factorization is used in a wide range of applications including statistics, optimization, and machine learning. To factor a given a matrix $M\in\mathbb R^{p\times q}$ of $\rank(M)=r$, we need to find size-$r$ vectors $a_1,...,a_p, b_1,...,b_q\in\mathbb R^r$ such that  $M_{ij} = \langle a_i, b_j\rangle$.

Often times, however, the matrix at hand as well as the elements in the factorization are required to have certain positivity structure~\cite{FP, GPT13, GPT15}. In statistical mixture models, for instance, we need to find a {\em nonnegative} factorization of the matrix at hand~\cite{CR93,GG12,KRS15,Vavasis09}. In other words, the vectors $a_i$ and $b_j$ need to be nonnegative. In the present article we study a more general type of factorization called positive semidefinite factorization. 
The vectors $a_i$ and $b_j$ in the decomposition are now replaced by $k\times k$ symmetric positive semidefinite matrices $A_i, B_j\in\mathcal S^k_+$, and $k$ is the size of the positive semidefinite factorization of $M$. Here the space of symmetric $k\times k$ matrices is denoted by $\mathcal S^k$, the cone of $k\times k$ positive semidefinite matrices by $\mathcal S^k_+$, and the inner product on $\mathcal S^k$ is given by 
$$
\langle A,B \rangle =\text{trace}(AB).
$$

\begin{definition}Given a matrix $M\in\mathbb R^{p\times q}_{\geq 0}$ with nonnegative entries, a {\em positive semidefinite (psd) factorization} of size $k$ is a collection of matrices $A_1,...,A_p, B_1,..., B_q\in\mathcal S^{k}_+$ such that $M_{ij} = \langle A_i,B_j\rangle$. The {\em positive semidefinite rank} {\em (psd rank)} of the matrix $M$ is the smallest $k \in \mathbb{N}$ for which such a factorization exists. It is denoted by $\rank_{psd}(M)$.
\end{definition}
The nonnegativity constraint on the entries of $M$ is natural here since for any two psd matrices $A, B\in\mathcal S^k_+$, it is always the case that $\langle A, B\rangle \geq 0$. To see this, write $A = UU^T, B = VV^T$ for some $U,V\in\mathbb R^{k\times k}$. Then, $\text{trace}(AB) = \text{trace}((V^TU)(V^TU)^T) \geq 0$ since $((V^TU)(V^TU)^T)$ is positive semidefinite. Thus, in order for $M$ to have finite psd rank, its entries need to be nonnegative.

Given a polytope $P$, the smallest number $k$ such that the polytope can be written as a projection of a linear slice of $\mathcal S^k_+$ is called the semidefinite extension complexity of $P$. This quantity is also equal to the psd rank of a slack matrix for the polytope $P$. This connection between positive semidefinite rank and semidefinite extension complexity is analogous to the connection between nonnegative rank and linear extension complexity, established in the seminal paper of Yannakakis~\cite{Yannakakis91}. This was the first paper in the line of work providing super-polynomial lower bounds on the linear and semidefinite extension complexities of families of polytopes~\cite{FMPTW,Rothvoss14,LRST,FSP15,LRS15}. The geometric aspects as well as many of the properties of psd rank have been studied in a number of recent articles~\cite{FGPRT,GPT13,GPT15,GPRT,GRT13,GRT}.
 
In this paper we study the space $\mathcal M_{r, k}^{p\times q}$ (or $\mathcal M_{r,k}$ for short) of $p \times q$ nonnegative matrices of rank at most $r$ and psd rank at most $k$. 
By Tarski-Seidenberg's Theorem \cite[Theorem 2.76]{basu2005algorithms} this set is semialgebraic, i.e. it is defined by finitely many polynomial equations and inequalities, or it is a finite union of such sets. It lies inside the variety $\mathcal V^{p\times q}_r$ (or $\mathcal V_r$ for short) of $p\times q$ matrices of rank at most~$r$. We study the geometry of $\mathcal M_{r, k}$, and in particular, we investigate the {\em boundary} $\mathcal \partial \mathcal  M_{r,k}$ of $\mathcal M_{r, k}$ as a subset of $\mathcal V_r$.

\begin{definition} The {\em topological boundary} of $\mathcal M_{r, k}$, denoted by $\partial \mathcal M_{r, k}$, is its boundary as a subset of $\mathcal V_{r}$. In other words, it consists of all matrices $M\in\mathcal V_{r}$ such that for every $\epsilon > 0$, the ball with radius $\epsilon$ and center $M$, denoted by $\mathcal B_{\epsilon}(M)$, satisfies the condition that $\mathcal B_{\epsilon}(M)\cap\mathcal V_r$ intersects $\mathcal M_{r, k}$ as well as its complement $\mathcal V_r\setminus \mathcal M_{r, k}$.
The {\em algebraic boundary} of $\mathcal M_{r, k}$, denoted by $\overline{\partial\mathcal M_{r, k}}$ is the Zariski closure of $\partial \mathcal M_{r, k}$ over~$\mathbb R$.
\end{definition}

In Section \ref{sec:psdrank2}, we completely describe $\partial \mathcal M^{p\times q}_{3, 2}$, as well as $\overline{\partial \mathcal M^{p\times q}_{3, 2}}$. More precisely,  Corollary \ref{cor:algebraic} shows that a matrix $M$ lies on the boundary $\partial\mathcal M^{p\times q}_{3, 2}$ if and only if in every psd factorization $M_{ij} = \langle A_i, B_j\rangle$, at least three of the matrices $A_1,\dots, A_p$ and at least three of the matrices $B_1,\dots, B_q$ have rank one.

In Sections~\ref{sec:geometricInterpretation} and~\ref{sec:higherPsdRank}, we study the general case $\partial\mathcal M^{p\times q}_{r, k}$.  Conjecture \ref{thm:k+1k} is an analogue of Corollary \ref{cor:algebraic}. It states that a matrix $M$ lies on the boundary $\partial \mathcal M^{p\times q}_{r, k}$ if and only if in every psd factorization $M_{ij} = \langle A_i, B_j\rangle$, at least $k+1$ of the matrices $A_1,\dots, A_p$ have rank one and at least $k+1$ of the matrices $B_1,\dots, B_q$ have rank one. In Section~\ref{sec:5.1}, we give theoretical evidence supporting this conjecture in the simplest situation where $p=q=r=k+1$. In Section \ref{section:higher_psd_rank}, we present computational examples.  Our code is available at
\[ \hbox{\tt https://github.com/kaiekubjas/psd-rank}\hspace{0.1cm} .\]

Our results are based on a geometric interpretation of psd rank, which is explained in Section \ref{sec:preliminaries}. Given a nonnegative matrix $M$ of rank $r$ satisfying $M \mathbf 1=\mathbf 1$, we can associate to it nested polytopes   $P\subseteq Q \subseteq \R^{r-1} $. Theorem \ref{thm:spectrahedron}, proved in \cite{GRT}, shows that $M$ has psd rank at most $k$ if and only if we can fit a projection of a slice of the cone of $k\times k$ positive semidefinite matrices $\mathcal{S}^k_+$
between $P$ and $Q$. When we restrict to the case when the rank of $M$ is three, this result states that $M$ has psd rank two if and only if we can nest an ellipse between the two nested polygons $P$ and $Q$ associated to $M$. In Theorem \ref{main_theorem} we show that $M$ lies on the boundary $\partial \mathcal M^{p\times q}_{3, 2}$ if and only if every ellipse that nests between the two polygons $P$ and $Q$, touches at least three vertices of $P$ and at least three edges of $Q$. The statement of Conjecture \ref{conjecture:geometric_description2} is analogous to the statement of Theorem \ref{main_theorem} for the general case~$\partial \mathcal M^{p \times q}_{r, k}$.

\subsection*{Acknowledgments}

Part of this work was done while the first and second authors were visiting the Simons Institute for the Theory of Computing, UC Berkeley. We thank Kristian Ranestad and Bernd Sturmfels for very helpful discussions, Rekha Thomas for reading the first draft of the article and Sophia Sage Elia for making Figure~\ref{fig:circulant_matrices_3D}.

\section{Preliminaries}\label{sec:preliminaries}

Many of the basic properties of psd rank have been studied in \cite{FGPRT}. We give a brief overview of the results used in the present article.
\subsection{Bounds}
The psd rank of a matrix is bounded below by the inequality $$\rank(M)\leq \binom{\rank_{\text{psd}}(M)+1}2$$ since one can vectorize the symmetric matrices in a given psd factorization and consider the trace inner product as a dot product. On the other hand, the psd rank is upper bounded by the nonnegative rank $$\rank_{psd}(M)\leq \rank_+(M)$$ since one can obtain a psd factorization from a nonnegative factorization by using diagonal matrices. The psd rank of $M$ can be any integer satisfying these inequalities.

\subsection{Geometric description}
\subsubsection*{From nested polytopes to nonnegative matrices}

We now describe the geometric interpretation of psd rank.
Let $P\subseteq \mathbb R^{r-1}$ be a polytope and $Q\subseteq\mathbb R^{r-1}$ be a polyhedron such that $P\subseteq Q$. Assume that $P = \conv\{v_1,...,v_p\}$ and $Q$ is given by the inequality representation $Q = \{x\in\mathbb R^{r-1} : h_j^Tx \leq z_j, j=1,...,q\}$, where $v_1,...,,v_p, h_1,...,h_q\in\mathbb R^{r-1}$ and $z_1,\dots, z_q\in\mathbb R$. The {\em generalized slack matrix} of the pair $P, Q$, denoted by $S_{P,Q}$, is the $p\times q$ matrix whose $(i,j)$-th entry is
$z_j - h_j^T v_i $. 

\begin{remark}
The generalized slack matrix depends on the representations of $P$ and $Q$ as the convex hull of finitely many points and as the intersection of finitely many half-spaces whereas the slack matrix depends only on $P$ and $Q$. We will abuse the notation and write $S_{P,Q}$ for the generalized slack matrix as by the next result the $\rank_{psd}(S_{P, Q})$ is independent of the representations of $P$ and $Q$.
\end{remark}

\begin{theorem}[Proposition~3.6 in~\cite{GRT}]\label{thm:spectrahedron} Let $P\subset \mathbb R^{r-1}$ be a polytope and $Q \subseteq \mathbb R^{r-1}$ a polyhedron such that $P\subseteq Q$. Then, $\rank_{psd}(S_{P, Q})$ is the smallest integer $k$ for which there exists an affine subspace $L$ of $\mathcal S^k$ and a linear map $\pi$ such that $P\subseteq \pi(L\cap\mathcal S^k_+) \subseteq Q$.
\end{theorem}

A {\em spectrahedron} of size $k$ is an affine slice of the cone $\mathcal S_+^k$ of $k\times k$ positive semidefinite matrices. A {\em spectrahedral shadow} of size $k$ is a projection of a spectrahedron of size $k$. Therefore, Theorem \ref{thm:spectrahedron} states that the matrix $S_{P, Q}$ has psd rank at most $k$ if and only if one can fit a spectrahedral shadow of size $k$ between $P$ and $Q$.

\begin{remark}
Given $M$, the polytopes $P$ and $Q$ are not unique, but the statement of Theorem~\ref{thm:spectrahedron} holds regardless of which pair $P, Q$ such that $M = S_{P, Q}$, is chosen. 
\end{remark}

\subsubsection*{From nonnegative matrices to nested polytopes}

Given a $p \times q$ nonnegative matrix $M$, we can assume that it contains no zero rows as removing zero rows does not change its psd rank. Secondly, we may assume that $\mathbf 1$ is contained in the column span of $M$ as scaling its rows by scalars also keeps the psd rank fixed. Consider a rank-size factorization $M=AB$ with $A$ having rows $A_i=(a_i^T,1)$. Let
$$
P=\text{conv}(a_1,\ldots,a_p) \text{ and } Q=\{x \in \R^{r-1}:(x^T,1)B \geq 0\}.
$$
Then $P \subseteq Q$ and $S_{P,Q}=M$.\

Without loss of generality, we may further assume that $M\mathbf 1 = \mathbf 1$ by scaling the rows of $M$ by its row sums. The following lemma shows that in this case we can choose $P$ and $Q$ to be bounded.

\begin{lemma}[Lemma~4.1 in~\cite{FGPRT}]\label{lem:nest} Let $M\in\mathbb R^{p\times q}_{\geq 0}$ be a nonnegative matrix and assume that $M\mathbf 1 = \mathbf 1$. Let $\rank(M) = r$. Then, there exist polytopes $P, Q\subseteq \mathbb R^{r-1}$ such that $P\subseteq Q$ and $M$ is the slack matrix of the pair $P, Q$.
\end{lemma}

\subsubsection*{The geometry of $\mathcal M_{r,k}^{p\times q}$}

A point $M \in \mathcal M^{p\times q}_{r, k}$ is an {\em interior point} of $\mathcal M^{p\times q}_{r, k}$ if there is an open ball $B_{\epsilon}(M) \subset R^{p \times q}$ that satisfies $B_{\epsilon}(M) \cap \mathcal V^{p\times q}_{r} = B_{\epsilon}(M) \cap \mathcal M^{p\times q}_{r, k}$.
 By the following lemma, we can check whether a matrix lies in the interior or boundary of $\mathcal M_{r,k}^{p\times q}$ by checking this for its rescaling that satisfies $M\mathbf 1=\mathbf1$.

\begin{lemma}\label{lem:rescale} A matrix $M\in\mathbb R^{p\times q}_{\geq 0}$ without zero rows lies in the interior of $\mathcal M_{r, k}$ if and only if  the matrix $N$, obtained from $M$ by rescaling such that $N\mathbf 1=\mathbf1$, lies in the interior of $\mathcal M_{r, k} \cap \{P \in \mathbb R^{p\times q}_{\geq 0} :P\mathbf 1 = \mathbf1\}$ with respect to $\mathcal V_r\cap  \{P \in \mathbb R^{p\times q}_{\geq 0} :P\mathbf 1 = \mathbf1\}$. 
\end{lemma}

\begin{proof}
First assume that the rescaled matrix $N$ lies in the interior of $\mathcal M_{r, k} \cap R$, where $R=\{P \in \mathbb R^{p\times q}_{\geq 0} :P\mathbf 1 = \mathbf1\}$. Thus, there exists $\epsilon > 0$ such that $\mathcal B_{\epsilon}(N)\cap\mathcal V_r\cap R \subseteq \mathcal M_{r, k} \cap R$.  Let $\alpha_1,\dots, \alpha_p$ be the row sums of $M$, i.e. $M\mathbf 1 = \alpha$. Without loss of generality, assume that $0<\alpha_1\leq \alpha_2\leq \cdots \leq \alpha_p$. Then, consider the ball $\mathcal B_{\epsilon\alpha_1}(M)$. If a matrix $M' =M + A\in \mathcal B_{\epsilon\alpha_1}(M)\cap \mathcal V_r$, then, after dividing the rows of $M'$ by $\alpha_1,\dots, \alpha_p$ respectively, we obtain the matrix $N + B$, where $B$ is the rescaled version of $A$. Since $\alpha_1\leq \cdots\leq \alpha_p$, then $\Vert B\Vert \leq \frac1{\alpha_1}\Vert A\Vert$. Thus $N+B\in\mathcal B_{\epsilon}(N)\cap\mathcal V_r\cap R\subseteq \mathcal M_{r, k}\cap R$. Since rescaling of the rows by positive numbers does not change the rank or psd rank, we have $M'\in\mathcal M_{r, k}$. Therefore, $\mathcal B_{\epsilon\alpha_1}(M)\cap\mathcal V_r\subseteq \mathcal M_{r, k}$, i.e. $M$ is in the interior of $\mathcal M_{r, k}$.

Now, assume that $M$ lies in the interior of $\mathcal M_{r, k}$. Then, there exists $\epsilon > 0$ such that $\mathcal B_{\epsilon}(M)\cap \mathcal V_{r}\subseteq \mathcal M_{r, k}$. Let $M\mathbf 1 = \alpha$, and assume that $0 < \alpha_1\leq \alpha_1\leq \cdots \leq \alpha_p$. Consider the ball $\mathcal B_{\epsilon / \alpha_p}(N)$. If $N' = N + B \in \mathcal B_{\epsilon / \alpha_p}(N) \cap \mathcal{V}_r \cap R$, then after multiplying the rows of $N'$ by $\alpha_1,\dots, \alpha_p$ respectively we obtain the matrix $M' = M + A$, where $A$ is the rescaled version of $B$, and $||A||\leq \alpha_p ||B||$. Thus, $M'\in\mathcal B_{\epsilon}(M)\cap \mathcal V_r\subseteq \mathcal M_{r,k}$. Since rescaling of the rows by positive numbers does not change the rank or the psd rank, we have $N'\in\mathcal M_{r,k}$. Thus, $\mathcal B_{\epsilon/\alpha_p}(N) \cap \mathcal{V}_r \cap  R \subseteq \mathcal M_{r, k} \cap R$, so $N$ lies in the interior of $\mathcal M_{r, k} \cap R$.
\end{proof}

Lemma \ref{lem:rescale} implies that if we want to study the topology of $\mathcal M_{r, k}$ as a subset of $\mathcal V_r$, we can restrict ourselves to the topology of the space $\mathcal M_{r, k}\cap\{P \in R^{p\times q}_{\geq 0}: P\mathbf 1 = \mathbf 1\}$ as a subset of $\mathcal V_r\cap\{P \in R^{p\times q}_{\geq 0}:P\mathbf 1=\mathbf 1\}$, and Lemma \ref{lem:nest} gives us a recipe for thinking of the elements of this space geometrically.

\subsection{Comparison with nonnegative rank}

Three different versions of nonnegative matrix factorizations appear in the literature: In~\cite{Vavasis09} Vavasis considered the exact nonnegative factorization which asks whether a nonnegative matrix $M$ has a nonnegative factorization of size   equal to its rank. The geometric version of this question asks whether one can nest a simplex between the polytopes $P$ and $Q$. 

In~\cite{GG12} Gillis and Glineur defined restricted nonnegative rank as the minimum value $r$ such that there exist $A \in \mathbb{R}_{\geq 0}^{p \times r}$ and $B \in \mathbb{R}_{\geq 0}^{r \times q}$ with $M=AB$ and $\rank(A)=\rank(M)$.  The geometric interpretation of the restricted nonnegative rank asks for the minimal $r$ such that there exist $r$ points whose convex hull can be nested between $P$ and $Q$.

The geometric version of the nonnegative rank factorization asks for the minimal $r$ such that there exist $r$ points whose convex hull can be nested between an $(r-1)$-dimensional polytope inside an $q$-simplex. These polytopes are not $P$ and $Q$ as defined in this paper. See~\cite[Theorem 3.1]{CR93} for details.

In the psd rank case there is no distinction between the psd rank and the restricted psd rank, because taking an intersection with a subspace does not change the size of a spectrahedral shadow while intersecting a polytope with a subspace can change the number of vertices. Conjecture~\ref{conjecture:spectrahedra_are_enough} also suggests that there is  no distinction between the spectrahedron and the spectrahedral shadow case which we can compare with simplices and polytopes in the nonnegative rank case, or equivalently the exact nonnegative matrix factorization and restricted nonnegative factorization case.

\section{Matrices of rank three and psd rank two}\label{section:rank3psdrank2}\label{sec:psdrank2}

In this section we study the set $\mathcal M_{3, 2}$ of matrices of rank at most three and psd rank at most two. We completely characterize its topological and algebraic boundaries $\partial \mathcal M_{3, 2}$ and $\overline{\partial\mathcal M_{3, 2}}$.

Consider a matrix $M\in\mathbb R^{p\times q}_{\geq 0}$ of rank three. We get a $2$-polytope $P$ and a $2$-polyhedron $Q$ such that $P \subseteq Q \subset \mathbb{R}^2$. Theorem \ref{thm:spectrahedron} now has the following simpler form.

\begin{corollary}[Proposition~4.1 in \cite{GRT}]\label{cor:psd_rank_two} Let $M$ be a nonnegative rank three matrix. Let $P\subseteq Q\subseteq \mathbb R^2$ be 
a polytope and a polyhedron for which $M = S_{P, Q}$. Then
$\rank_{psd}(M) = 2$ if and only if there exists a half-conic such that its convex hull $C$
satisfies $P \subseteq C \subseteq Q$. In particular if $Q$ is bounded, then $\rank_{psd}(M) = 2$  if and
only if we can fit an ellipse between $P$ and $Q$.
\end{corollary}

Half-conics are ellipses, parabolas and connected components of hyperbolas in $\R^2$. If $M\mathbf 1 = \mathbf 1$, then $P$ and $Q$ are bounded and the half-conic in Corollary~\ref{cor:psd_rank_two}  is an ellipse. Using this geometric interpretation of psd rank two, we give a condition on when a matrix $M$ lies in the interior of $\mathcal M_{3, 2}$.

\begin{lemma}\label{continuity_of_factorizations}
Let $M \in \R^{p \times q}$ be such that $M \mathbf 1 =\mathbf 1$ and $\rank(M)=r$. In a small neighborhood of $M$, there exists a continuous map $\mathcal{V}_r \cap \{M \in \R^{p \times q}:M \mathbf 1 = \mathbf 1\} \rightarrow \R^{p \times r} \times \R^{r \times q}, M \mapsto (A,B)$ such that $M=AB$ and the last column of $A$ consists of ones.
\end{lemma}

\begin{proof}
Let $\rank (M)=r$. Consider the rank-size factorization $M=AB$ where $A$ consists of $r-1$ linearly independent columns of $M$ and the column $\mathbf{1}$ such that $\mathbf{1}$ is not in the column span of the $r-1$ columns. Then the entries of $B$ are solutions of the linear system of equations $AB=M$. In particular, we can choose $r$ linearly independent rows of $M$ and write down the square system corresponding to the rows. Then each entry of $B$ is of the form $\frac{\det(\cdot)}{\det(\cdot)}$, where the upper determinant is in the entries of $A,M$ and the lower determinant is in the entries of $A$. However, the entries of $A$ are also entries of $M$. Hence, we have constructed a map that is continuous in the neighborhood of $M$ where the set of linearly independent columns and rows used for constructing $A$ and $B$ remain linearly independent.
\end{proof}

\begin{lemma}\label{lem:interior} Let $M$ be a nonnegative matrix of rank three satisfying $M \mathbf 1 =\mathbf 1$ such that there exist nested polytopes $P$ and $Q$ for which $M = S_{P, Q}$. Then  $M$ lies in the interior of $\mathcal M_{3, 2}$ if and only if there exists a region $E$ bounded by an ellipse such that $P\subset E\subset Q$ and the boundary of $E$ does not contain any vertices of $P$.
\end{lemma}

\begin{proof}
 By Lemma~\ref{lem:rescale}, we may assume throughout the proof that $M \mathbf 1=\mathbf 1$ and hence $P \subseteq Q$ are bounded. Abusing the terminology, we will call the region bounded by an ellipse an ellipse in this proof.

Assume first that $M$ lies in the interior of $\mathcal M_{3, 2}$. By Lemma~\ref{lem:nest} and Corollary~\ref{cor:psd_rank_two} there exists an ellipse $E$ such that $P\subseteq E\subseteq Q$. If the boundary of $E$ does not contain any vertices of $P$, then we are done. Suppose that the boundary of $E$ contains some vertices of $P$. We are going to find another ellipse $E'$ such that $P\subset  E'\subset Q$ and the boundary of $E'$ does not contain any vertices of $P$. 

Since $M$ is in the interior of $\mathcal M_{3, 2}$, none of the entries of $M$ are 0, so the boundary of the polygon $Q$ does not contain any vertices of $P$. Moreover, there exists $\epsilon > 0$ such that $\mathcal V_3\cap \mathcal B_{\epsilon}(M)\subset \mathcal M_{3, 2}$. Pick a point in the interior of the polygon $P$ and consider the polygon $tP$ obtained by a homothety centered at the selected point with some $t>1$. Then, $P\subset tP\subseteq Q$ for a small enough $t>1$, and $P$ is strictly contained in $tP$. Now consider the generalized slack matrix of $tP$ and $Q$ and call it $M_t$. We can choose $t$ close enough to 1 so that $M_t\in\mathcal B_{\epsilon}(M)\subseteq \mathcal M_{3,2}$. Thus, $M_t$ has psd rank at most two and there exists an ellipse $E'$ such that $tP\subset E'\subset Q$. Therefore $P \subset tP \subset E'\subset Q$ and the boundary of the ellipse $E'$ does not contain any vertices of $P$.
 
Now suppose that there exists an ellipse $E$ and polygons $P$ and $Q$ such that $P\subset E\subset Q$ and the ellipse $E$ does not contain any vertices of $P$. It is possible to shrink the ellipse $E$ slightly so that it also does not touch any edges of $Q$ either. We obtain an ellipse $E'$ that does not touch any vertices of $P$ and does not touch any edges of $Q$. By Lemma~\ref{continuity_of_factorizations}, for any matrix $M'\in\mathcal B_{\epsilon}(M)\cap\mathcal V_3 \cap \{M \in \R^{p \times q}:M\mathbf 1 = \mathbf 1\}$ we obtain polyhedra that are small perturbations of $P$ and $Q$ and hence $E'$ is nested between them. Therefore, $M'\in\mathcal M_{3, 2}$ and so $\mathcal B_{\epsilon}(M)\cap\mathcal V_3 \cap \{M \in \R^{p \times q}:M\mathbf 1 = \mathbf 1\} \subseteq \mathcal M_{3,2}$.
\end{proof}

We can now show how $\mathcal M_{3, 2}$ relates to the variety $\mathcal V_3$.
\begin{proposition}
The Zariski closure of $\mathcal M_{3,2}^{p\times q}$ over the real numbers is  $\mathcal{V}^{p \times q}_3$.
\end{proposition}

\begin{proof}
Suppose that there exists a ball $\mathcal B\subseteq \mathbb R^{p\times q}$ such that $\mathcal B\cap \mathcal{V}_3 \subseteq \mathcal M_{3, 2}$. This implies that the dimension of $\mathcal M_{3,2}^{p\times q}$ is equal to that of $\mathcal{V}_3$, and since $\mathcal M_{3,2} \subset \mathcal{V}_3$ and $\mathcal{V}_3$ is irreducible \cite[Theorem 2.10]{bruns1988determinantal}, the Zariski closure of $\mathcal M_{3, 2}$ over the real numbers equals $\mathcal{V}_3$.

We show how to find such a ball $\mathcal B$. By Lemmas~\ref{continuity_of_factorizations} and~\ref{lem:interior}, it would suffice to find nested polygons $P\subseteq Q\subseteq \mathbb R^2$ such that $P$ has $p$ vertices, $Q$ has $q$ edges and there exists an ellipse nested between them that does not touch the vertices of $P$. Such a configuration certainly exists, for example, we can consider a regular $p$-gon $P$ centered at the origin with length $1$ from the origin to any of its vertices, and a regular $q$-gon $Q$ centered at the origin with length $5$ from the origin to any of its edges. Then, we can fit a circle of radius $2$ and center the origin between $P$ and $Q$ so that it does not touch the vertices of $P$.
\end{proof}

\begin{remark}
The set of $p \times q$ matrices of psd rank at most $k$ is connected as it is the image under the parametrization map of the connected set $(\mathcal{S}^k_+)^p \times (\mathcal{S}^k_+)^q$. If we also fix the rank, then it is not known if the corresponding set is connected.
\end{remark}

The following theorem is the main result of this section.
 
\begin{theorem}\label{main_theorem}
We describe the topological and algebraic boundaries of $\mathcal M_{3,2}^{p\times q}$.
\begin{enumerate}
\item[a.] A matrix $M\in \mathcal M_{3,2}^{p\times q}$ satisfying $M \mathbf 1=\mathbf 1$  lies on the topological boundary $\partial\mathcal M_{3,2}^{p\times q}$ if and only if $M_{ij}=0$ for some $i,j$, or each ellipse  that fits between the polygons $P$ and $Q$ contains at least three vertices of the inner polygon $P$ and is tangent to at least three edges of the outer polygon $Q$.
\item[b.] A matrix $M\in\overline{\mathcal M_{3, 2}^{p\times q}} = \mathcal V_3^{p\times q}$ satisfying $M \mathbf 1=\mathbf 1$ lies on the algebraic boundary $\overline{\partial\mathcal M_{3, 2}^{p\times q}}$ if and only if $M_{ij} = 0$ for some $i,j$ or there exists an ellipse that contains at least three vertices of $P$ and is tangent to at least three edges of $Q$.
\item[c.] The algebraic boundary of $\mathcal M_{3, 2}^{p\times q}$ is the union of $\binom p3\binom q3+pq$ irreducible components. Besides the $pq$ components $M_{ij}=0$, there are $\binom p3\binom q3$ components each of which is defined by the $4\times 4$ minors of $M$ and one additional polynomial equation with $1035$ terms homogeneous of degree $24$ in the entries of $M$ and homogeneous of degree $8$ in each row and each column of a $3\times 3$ submatrix of $M$.
\end{enumerate}
\end{theorem}

\begin{proof}
Let $\tilde{P}$ and $\tilde{Q}$ be the projective completions of $\text{cone}(P \times \{1\})$ and $\text{cone}(Q \times \{1\})$, i.e. the closures of images of $\text{cone}(P \times \{1\})-\{0\}$ and $\text{cone}(Q \times \{1\})-\{0\}$ under the map $\R^3 \rightarrow \mathbb P^2, (x,y,z) \mapsto [x:y:z]$. In \cite{Gallier}, $\tilde{P}$ and $\tilde{Q}$ are called projective polyhedra. If $P$ and $Q$ are bounded, there is no need to take closure. Hence, in this case there is one-to-one correspondence between statements about incidence relations in the affine and projective case.  In Section~\ref{sec:preliminaries}, we required $A$ to have rows $A_i=(a_i^T,1)$ and defined $P=\text{conv}(a_1,\ldots,a_p)$. Similarly, the last row of $B$ gave constant terms of inequalities defining $Q$. Thus $\text{cone}(P \times \{1\})$ is the cone over the rows of $A$ and $\text{cone}(Q \times \{1\})=\{x \in \R^3:x^T B \geq 0\}$. This allows us to define $\tilde{P}$ and $\tilde{Q}$ for general $M$ (even if $\mathbf 1$ is not in the column span of $M$). Since in the projective plane all non-degenerate conics are equivalent, we will use the word ``conic" instead of ``ellipse". Abusing the terminology, we will also call the region bounded by a nondegenerate conic a conic in this proof. The region bounded by a nondegenerate conic is determined by the region bounded by the corresponding double cone in $\R^3$.

$(a)$ 
\underline{Only if:} We show the contrapositive of the statement: If all the entries of $M$ satisfying $M \mathbf 1=\mathbf 1$ are positive and there is a conic between $\tilde P$ and $\tilde Q$ whose boundary contains at most two vertices of $\tilde P$ or is tangent to at most two edges of $\tilde Q$, then $M$ lies in the interior of $M_{3,2}^{p\times q}$.

First, if there is a conic $E$ between $\tilde P$ and $\tilde Q$ whose boundary touches neither of the polytopes, then $M$ is in the interior of $\mathcal M_{3,2}$ by Lemma \ref{lem:interior}.
If at most two edges of $\tilde Q$ are tangent to the boundary of the conic $E$, then $\tilde P \subset E \subset \tilde Q$ can be transformed by a projective transformation such that the two tangent edges are $x = 0$ and $y = 0$ and that the points of tangency are $[0:1:1]$ and $[1:0:1]$. We denote the image of $E$ by $\overline{E}$.
The equation of the conic $\overline{E}$ has the form $ax^2 + bxy + cy^2 + dxz+eyz+fz^2 = 0$. We know that the only point that lies on the conic $\overline{E}$ with $x=0$ is the point $[0:1:1]$ since $\overline{E}$ touches the line $x=0$ at $[0:1:1]$. If we plug in $x=0$, we get
$$cy^2 + eyz+fz^2 = 0.$$
We may assume $c \geq 0$, hence we must have $cy^2 + eyz + fz^2 = (y-z)^2$.  Therefore, $c=1, e=-2, f=1$. Similarly, since $\overline{E}$ touches the line $y=0$ at $[1:0:1]$, when we plug in $y=0$, we get that $ax^2 + dxz+fz^2 = (x-1)^2$, so, $a=1, d=-2, f=1$. Thus, the conic $\overline{E}$ has the form
$$\{(x, y): x^2 + bxy + y^2 -2xz-2yz + z^2 = 0\},$$
for some $b$. The conic is degenerate if and only if $b=2$. Since $E$ is nondegenerate, also $\overline{E}$ is nondegenerate. The double cone corresponding to $\overline{E}$ in $\R^3$ is defined by $x^2 + bxy + y^2 -2xz-2yz + z^2 \leq 0$. Since $x=0$ and $y=0$ are tangent to this double cone and touch it at the points $(0,1,1)$ and $(1,0,1)$, for all nonzero $x$ and $y$ we have $xy >0$. This corresponds to $b<2$. For a slightly smaller value of $b$, we obtain a slightly larger double cone. The nondegenerate conic $\overline{E}' \subseteq \mathbb{P}^2$ corresponding to this double cone contains $\overline{E}$ and touches $\overline{E}$ only at the points $[1:0:1]$ and $[0:1:1]$. Let $E'$ be the preimage of $\overline{E}'$ under the projective transformation considered above. We have $\tilde P\subseteq E\subset E'\subseteq \tilde Q$ and the conic $E'$ does not touch $\tilde P$. Thus, by Lemma \ref{lem:interior}, $M$ lies in the interior of $\mathcal M_{3,2}$. The case when $E$ goes through at most two vertices of $\tilde P$ follows by duality. 

\underline{If:} By Lemma \ref{lem:interior}, if $M\in \mathcal M_{3,2}$ satisfying $M \mathbf 1=\mathbf 1$ lies in the interior, then there is a conic between $\tilde P$ and $\tilde Q$ that does not touch $\tilde P$. Thus, if every conic nested between $\tilde P$ and $\tilde Q$ contains at least three vertices of $\tilde P$ and touches at least three edges of $\tilde Q$, then $M$ lies on the boundary $\partial\mathcal M_{3, 2}$

 $(b),(c)$ 
If $M \in \mathbb{R}^{p \times q}$ without nonnegativity constraints satisfies $M \mathbf 1=\mathbf 1$, then one can define polytopes $P$ and $Q$ as explained before Lemma~\ref{lem:nest}. The difference is that $P \subseteq Q$ does not hold anymore, and we also might not have $\tilde P \subseteq \tilde Q$. Nevertheless, one can talk about vertices of $\tilde P$ and edges of $\tilde Q$.
Hence given three points $a,b,c$ in $\mathbb{P}^2$
and three lines $d,e,f$ in $\mathbb{P}^2$, each given by three homogeneous
coordinates, we seek the condition that there exists a conic $X$
such that $a,b,c$ lie on $X$ and $d,e,f$ are tangent to $X$.

Let $X=\begin{bmatrix} x_{11} & x_{12} & x_{13} \\ x_{12} & x_{22} & x_{23} \\ x_{13} & x_{23} & x_{33} \end{bmatrix}$ be the matrix of a conic.  Then the corresponding conic goes through the points $a,b,c$ if and only if 
\begin{align}\label{equations_for_points}
a^TXa=b^TXb=c^TXc=0.
\end{align}
 Similarly, the lines $d,e,f$ are tangent to the conic if and only 
\begin{align}\label{equations_for_lines}
d^TYd=e^TYe=f^TYf=0,
\end{align}
 where $XY=I_3$. We seek to eliminate the variables $X$ and $Y$. 

Let $[a,b,c]$ denote the matrix whose columns are $a,b,c$. First we assume that $[a,b,c]$  is the $3 \times 3$-identity matrix.
Then we proceed in two steps:

1) The equations (\ref{equations_for_points}) imply that $x_{11},x_{22},x_{33}$ are zero. We make the corresponding replacements in equations~(\ref{equations_for_lines}).

2) We use~\cite[formula (4.5) on page 48]{Sturmfels02} to get the resultant of three ternary quadrics to get a single polynomial in the entries of $d,e,f$.

Now we use invariant theory to obtain the desired polynomial in the general case. Let $g \in \textrm{GL}_3(\mathbb{R})$. The conic $X$ goes through the points $a,b,c$ and touches the lines $d,e,f$ if and only if the conic $g^{-T}Xg^{-1}$ goes through the points $ga,gb,gc$ and touches the lines $g^{-T}d,g^{-T}e,g^{-T}f$. Thus our desired polynomial belongs to the ring of invariants $\mathbb{R}[V^3 \oplus V^{*3}]^{\textrm{GL}_3(\mathbb{R})}$ where $V=\mathbb{R}^3$ and the action of $\textrm{GL}_3(\mathbb{R})$ on $V^3 \oplus V^{*3}$ is given by
$$g\cdot(a,b,c,d,e,f):=(ga,gb,gc,g^{-T}d,g^{-T}e,g^{-T}f).$$
 The First Fundamental Theorem states that $\mathbb{R}[V^3 \oplus V^{*3}]^{\textrm{GL}_3(\mathbb{R})}$ is generated by the bilinear functions $(i|j)$ on $V^3 \oplus  V^{*3}$ defined by 
$$(i|j):(a,b,c,d,e,f) \mapsto ([a,b,c]^T [d,e,f])_{ij}.$$ 
For the FFT see for example~\cite[Chapter 2.1]{KP}. In the special case when $[a,b,c]$ is the $3 \times 3$ identity matrix, $(i|j)$ maps to the $(i,j)$-th entry of $[d,e,f]$. Hence to obtain the desired polynomial in the general case, we replace in the resultant obtained in the special case the entries of the matrix $[d,e,f]$ by the entries of the matrix $[a,b,c]^T [d,e,f]$. 

\texttt{Maple} code for doing the steps in the previous paragraphs can be found at our website. This program outputs one polynomial of degree $1035$ homogeneous of degree $8$ in each of the rows and the columns of the matrix $\begin{bmatrix}-&a&-\\-&b&-\\-&c&-\end{bmatrix}\begin{bmatrix}| & | & |\\d & e & f \\| & | & |\end{bmatrix}$.  By construction, if this homogeneous polynomial vanishes and the projective polyhedron $\tilde P$ with vertices $a,b,c$ lies inside the projective polyhedron $\tilde Q$ with edges $d,e,f$ and $a,b,c,d,e,f$ are real, then there exists a conic nested between $\tilde P$ and $\tilde Q$ touching $d,e,f$ and containing $a,b,c$. Therefore, the Zariski closure of the condition that the only possible conics that can fit between $\tilde P$ and $\tilde Q$ touch at least three edges of $\tilde Q$ and at least three vertices of $\tilde P$ is exactly that there exists an conic that touches at least three edges of $\tilde Q$ and at least three vertices of $\tilde P$. none of the vertices of $\tilde P$ has the last coordinate equal to zero This proves $(b)$. 

To prove $(c)$, let $M\in\mathcal V_3$ be such that $M = A B$ and $a, b, c$ are three of the rows of $A$ and $d, e, f$ are three of the columns of $B$.  Then, the above-computed polynomial contains variables only from the entries of a $3\times 3$ submatrix of $M$ corresponding to these rows and columns. We can drop the assumption $M \mathbf 1 =\mathbf 1$ here: Scaling a row of $M$ by a constant corresponds to scaling the corresponding row of $A$ by the same constant, which does not influence equations~(\ref{equations_for_points}). For each three rows and three columns of $M$ we have one such polynomial, so the algebraic boundary is given by the union over each three rows and three columns of $M$ of the variety defined by the $4\times 4$ minors of $M$ and the corresponding degree $24$ polynomial with $1035$ terms.
\end{proof}

Here is an algebraic version of Theorem \ref{main_theorem}.

\begin{corollary} \label{cor:algebraic} A matrix $M\in\mathbb R^{p\times q}_{\geq 0}$ satisfying $M \mathbf 1=\mathbf1$ lies on the boundary $\partial \mathcal M_{3, 2}$ if and only if for every size-2 psd factorization $M_{ij} = \langle A_i, B_j\rangle$, at least three of the matrices $A_1,\dots, A_p\in\mathcal S^2_+$ have rank one and at least three of the matrices $B_1,\dots, B_q\in\mathcal S^2_+$ have rank one.
\end{corollary}

\begin{proof}

Suppose that $M\not\in\partial \mathcal M_{3,2}$. Let $P = \cone\{a_1,\dots, a_p\}$ and $Q = \{x \in \R^{r-1}: \langle x,b_j \rangle \geq 0 \text{ for } j=1,\ldots,q\}$ such that $M=S_{P,Q}$. By~\cite[Proposition 4.4]{GRT} and Theorem~\ref{main_theorem}, there exists an invertible linear map $\pi$ such that $P \subseteq \pi(\mathcal S_+^2) \subseteq Q$ and the boundary of $\pi(\mathcal S_+^2)$ contains at most two rays of $P$ or is tangent to at most two facets of $Q$.

The invertibility of $\pi$ gives
$$\pi^{-1}(P) \subseteq  \mathcal S_+^2 \subseteq \pi^{-1}(Q),$$
where $\pi^{-1}(P) =  \cone\{\pi^{-1}(a_1),\dots, \pi^{-1}(a_p)\}$ and 
$$\pi^{-1}(Q) = \{x\in L\cap \mathcal S^2 : \langle \pi(x),b_j \rangle \geq 0\}= \{x\in L\cap \mathcal S^2 : \langle  x,\pi^T(b_j)\rangle \geq 0\}.$$
Thus $M = S_{\pi^{-1}(P), \pi^{-1}(Q)}$, since 
$$M_{ij} = \langle  a_i,b_j\rangle =  \langle  \pi(\pi^{-1}(a_i)),b_j\rangle =  \langle   \pi^{-1}(a_i),\pi^T(b_j)\rangle.$$ 

The inclusion $\pi^{-1}(P) \subseteq  \mathcal S_+^2$ implies that $\pi^{-1}(a_1),\ldots,\pi^{-1}(a_p)$ are psd. Taking dual of the inclusion $\mathcal S_+^2 \subseteq \pi^{-1}(Q)$ gives that $\pi^T(b_1),\ldots,\pi^T(b_q)$ are psd. Since $\pi$ is invertible, we know that either the boundary of $\mathcal S^2_+$ contains at most two rays of $\pi^{-1}(P)$ or is tangent to at most two facets of $\pi^{-1}(Q)$. Hence $\pi^{-1}(a_1),\ldots,\pi^{-1}(a_p),\pi^T(b_1),\ldots,\pi^T(b_q)$ gives a psd factorization of $M$ with at most two of $\pi^{-1}(a_1),\ldots,\pi^{-1}(a_p)$ having rank one or at most two of $\pi^T(b_1),\ldots,\pi^T(b_q)$ having rank one.

Suppose that there exists a psd factorization of $M$, given by matrices $A_1,\dots, A_p, B_1,\dots, B_q\in\mathcal S_+^2$, such that at most two of the $A_i$ have rank one. 
Consider $P = \cone\{A_1,\dots, A_p\}$ and $Q = \{x\in \mathcal S^2 : \langle x,B_j\rangle \geq 0, \forall j=1,\dots, q\}$. Then $P \subseteq \mathcal S_+^2 \subseteq Q$ and the boundary of $\mathcal S_+^2$ contains at most two rays of $P$. Using the inner product preserving bijection between $\mathcal S^2$ and $\R^3$, we can consider all objects in $\R^3$. In particular, the images of $A_1,\dots, A_p, B_1,\dots, B_q$ in $\R^3$ give a rank factorization of $M$.  By Theorem~\ref{main_theorem} (a), we have $M\not\in\partial \mathcal M_{3,2}$.
\end{proof}

We now investigate the topological boundary more thoroughly.

\begin{proposition}
Suppose $M \in \sM^{p\times q}_{3,2}$ satisfying $M \mathbf 1 = \mathbf 1$ is strictly positive.  Then $M$ lies on the topological boundary if and only if there exists a unique ellipse that fits between $P$ and $Q$.
\end{proposition}

\begin{proof}
Abusing the terminology, we will call the region bounded by an ellipse an ellipse in this proof. A matrix in the relative interior of $\sM_{3,2}$ will have multiple ellipses nested between $P$ and $Q$: By the only if direction of the proof of Theorem~\ref{main_theorem} part (a), there exists an ellipse that is contained in $Q$ and strictly contains $P$. We can just take slight scalings of this ellipse to get multiple ellipses.  This proves the ``if'' direction.

For the ``only if'' direction, suppose $M$ lies on the topological boundary and $E_0$ and $E_1$ are two ellipses nested between $P$ and $Q$.  Let $E_{1/2}$ be the ellipse determined by averaging the quadratics defining $E_0$ and $E_1$, i.e.
\[ E_{1/2} = \left\{ x : q_0(x) + q_1(x) \geq 0 \right\} \textup{ where } E_i = \left\{ x : q_i(x) \geq 0 \right\}. \]
It is straightforward to see that $E_{1/2}$ is nested between $P$ and $Q$.  Furthermore, if $v$ is a vertex of $P$, then $E_{1/2}$ passes through $v$ if and only if both $E_0$ and $E_1$ pass through $v$.  Similarly, if $f$ is a facet of $Q$, then $E_{1/2}$ is incident to $f$ if and only if $E_0$ and $E_1$ are tangent to $f$ at the same point.  By Theorem~\ref{main_theorem}, the ellipse $E_{1/2}$ must pass through three vertices of $P$ and three edges of $Q$.  Hence, there must exist six distinct points that both $E_0$ and $E_1$ pass through. No three of the six points are collinear, since ellipses $E_0$ and $E_1$ pass through them. Since five distinct points in general position determine a unique conic, we must have that $E_0 = E_1$.
\end{proof}

\begin{example} \rm
In the previous result, we examined the geometric configurations on the boundary of the 
semialgebraic set coming from strictly positive matrices. The 
simplest idea for such a matrix is to take two equilateral triangles 
and expand the inner one until we are on a boundary configuration as in Figure~\ref{fig:two_equilateral_on_the_boundary}.

\begin{figure}[H]
\centering
\begin{subfigure}[b]{0.45\textwidth}
\centering
\includegraphics[width=0.7\textwidth]{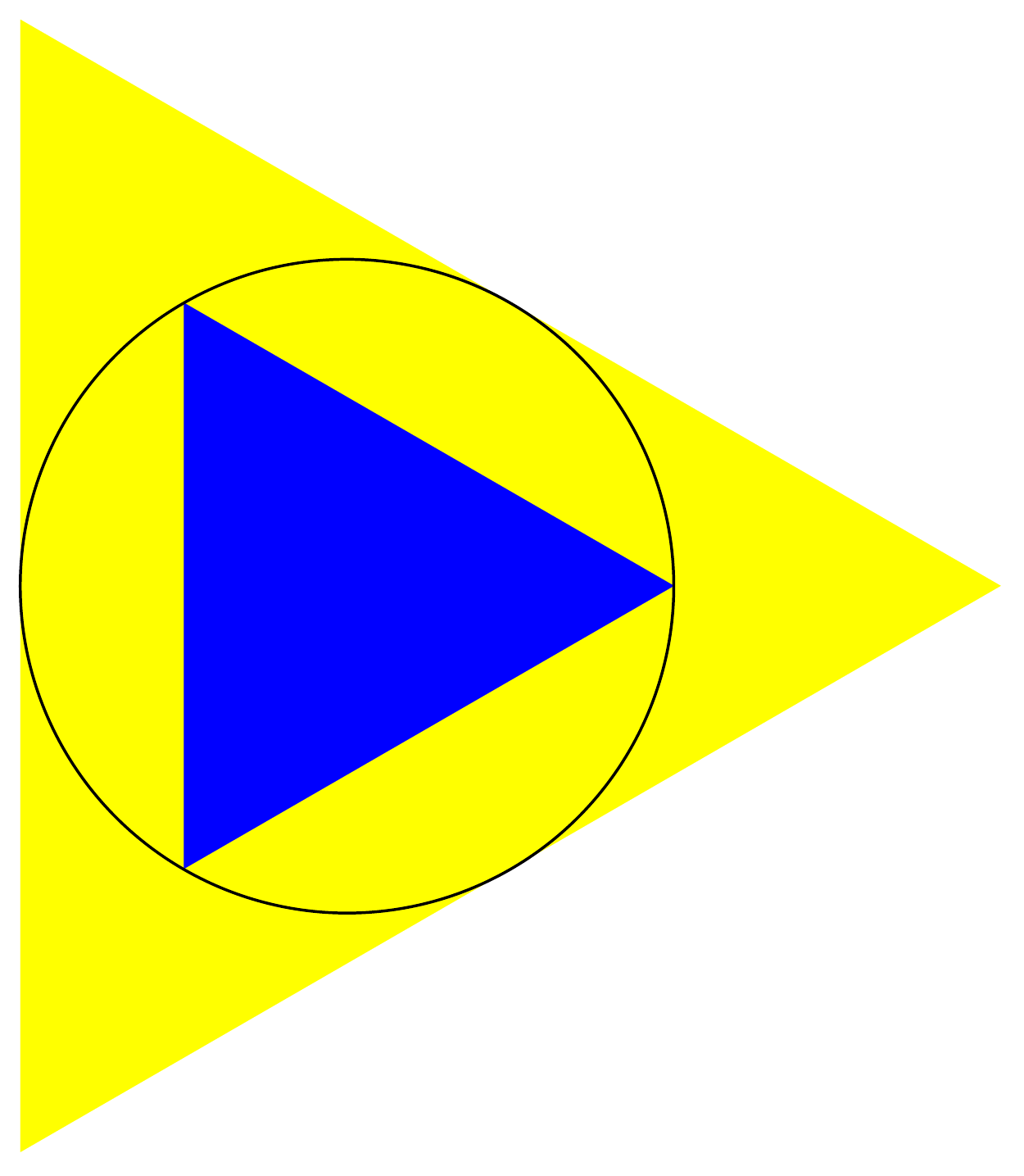}
\caption{Boundary configuration\\ \vspace{0.6cm} }
\label{fig:two_equilateral_on_the_boundary}
\end{subfigure}
\begin{subfigure}[b]{0.42\textwidth}
\centering
\includegraphics[width=0.7\textwidth]{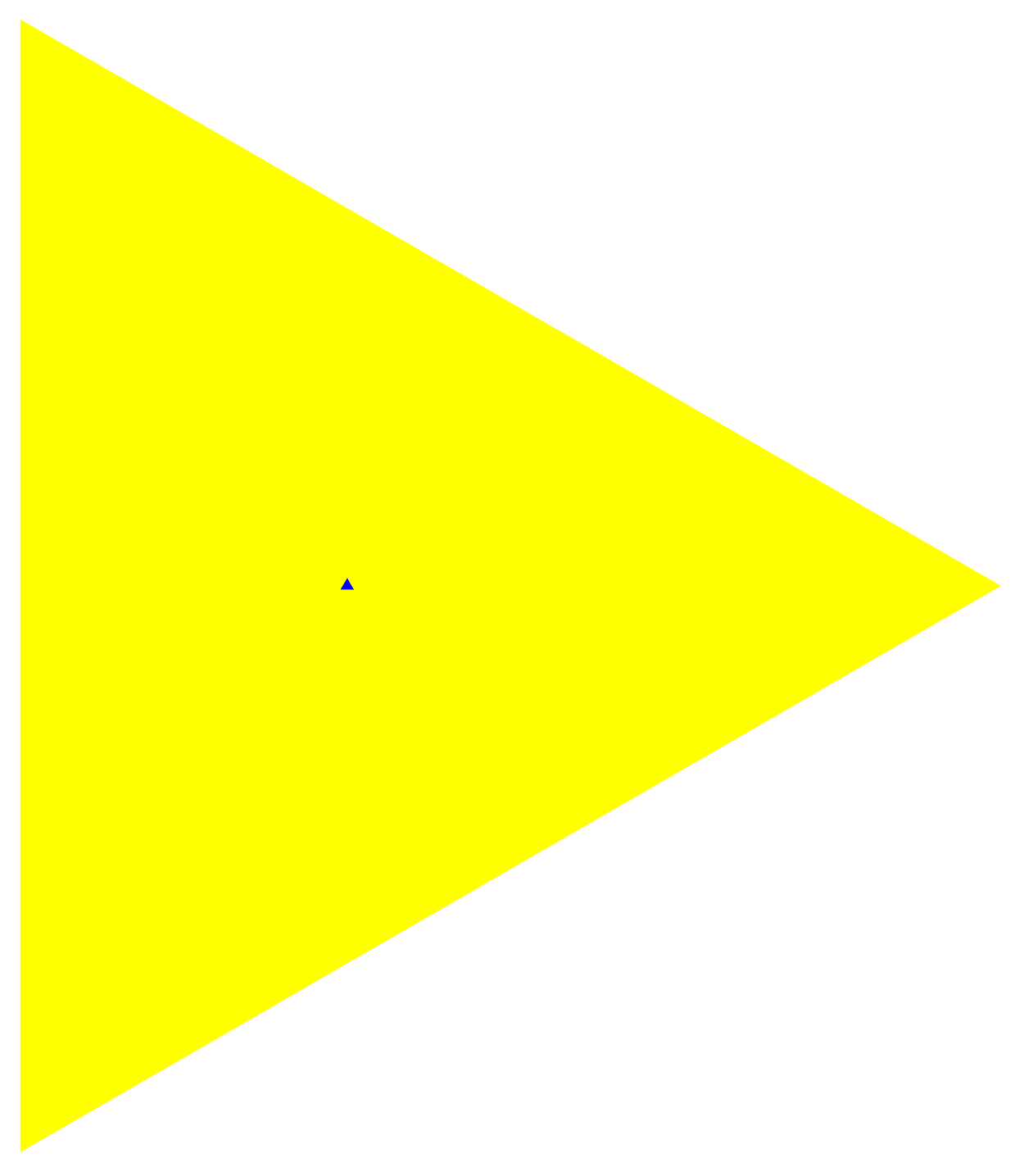}
\caption{Interior configuration which also lies on the algebraic boundary $\overline{\partial\mathcal M_{3, 2}}$}
\label{fig:not_on_the_boundary}
\end{subfigure}
\caption{Geometric configurations of matrices in $\mathcal M_{3,2}^{3\times 3}$}\label{fig:boundary_and_interior}
\end{figure}

This configuration has the slack matrix 
\begin{equation}\label{circulant_matrix_on_the_boundary}
\frac{1}{6}
\begin{bmatrix}
4 &1 & 1\\
1 & 4 &1\\
1 & 1 & 4
\end{bmatrix}.
\end{equation}
The $1035$ term boundary polynomial from Theorem~\ref{main_theorem} vanishes on this matrix, as we expect.

This matrix lies in the set of $3 \times 3$ circulant matrices which have the form 
$$
\begin{bmatrix}
a & b & c\\
c & a & b\\
b & c & a
\end{bmatrix}.
$$
It was shown in~\cite[Example 2.7]{FGPRT} that these matrices have psd rank at most two precisely when $a^2+b^2+c^2-2(ab+ac+bc) \leq 0$. As expected, whenever this polynomial vanishes, the $1035$ term boundary polynomial vanishes as well. The matrix~(\ref{circulant_matrix_on_the_boundary}) is a regular point of the hypersurface defined by the boundary polynomial.
Figure~\ref{fig:not_on_the_boundary} shows an instance of parameters $a,b,c$ such that the matrix is on the algebraic boundary but not on the topological boundary -- the polynomial vanishes, but the matrix lies in the interior of $\mathcal M_{3, 2}$.

We were interested in finding out if the $1035$ term boundary polynomial could be used in an inequality to classify circulant matrices of psd rank at most two. The family of circulant matrices which have $c=1$ and whose psd rank is at most two is depicted in Figure~\ref{fig:circulantBoundary2D}. The boundary polynomial, shown in Figure~\ref{fig:regionsBoundaryPolynomial2D}, takes both positive and negative values on the interior of the space. Figures~\ref{fig:circulantBoundary3D} and~\ref{fig:regionsBoundaryPolynomial3D} show the semialgebraic set and the boundary polynomial in the $3$-dimensional
space.

\begin{figure}[h!]
\centering
\begin{subfigure}[b]{0.45\textwidth}
\includegraphics[width=0.9\textwidth]{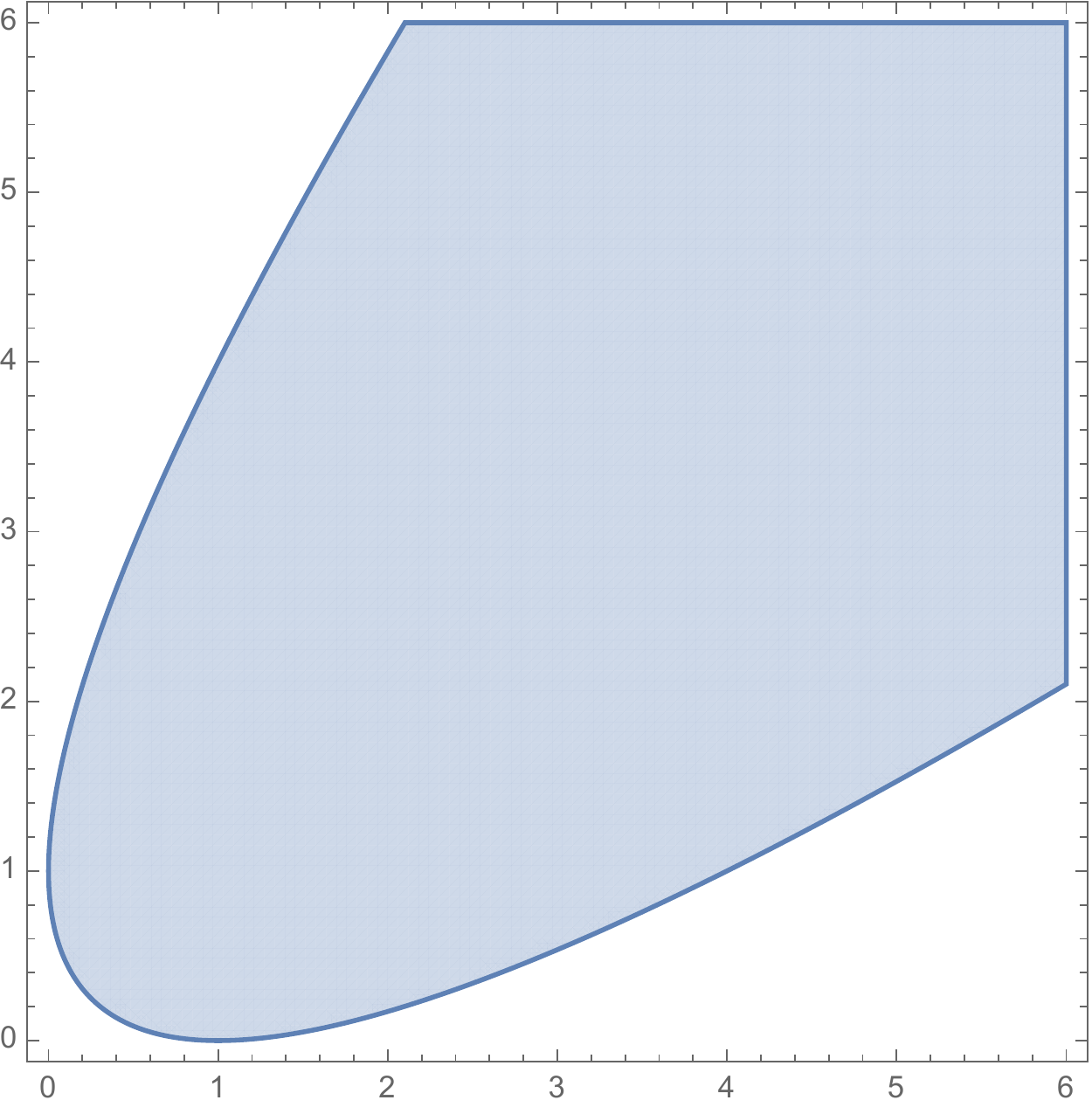}
\caption{Circulant matrices of psd rank at most 2}
\label{fig:circulantBoundary2D}
\end{subfigure}
\begin{subfigure}[b]{0.45\textwidth}
\centering
\includegraphics[width=0.9\textwidth]{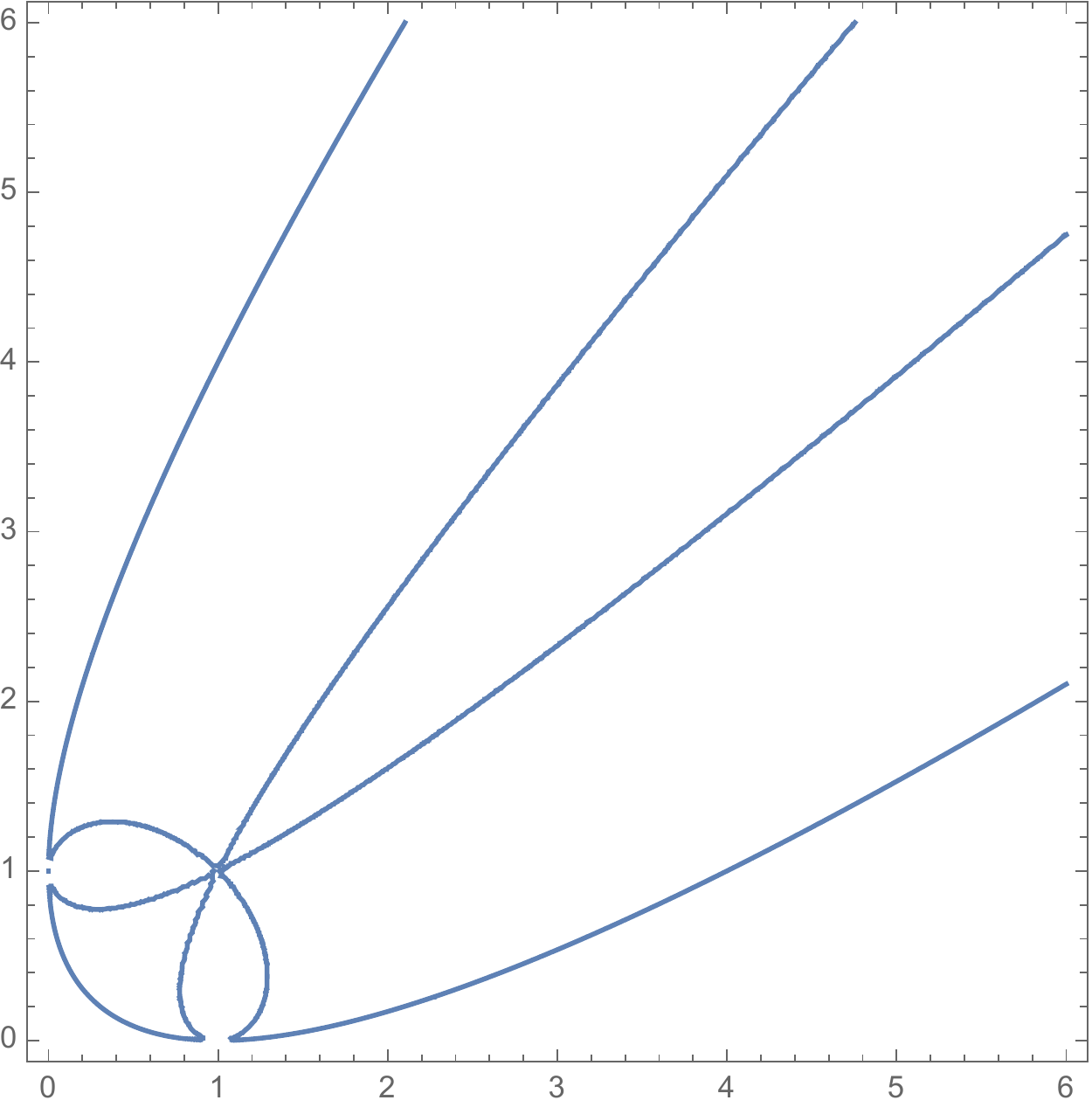}
\caption{The boundary polynomial}
\label{fig:regionsBoundaryPolynomial2D}
\end{subfigure}
\caption{$3 \times 3$ circulant matrices in $\mathbb{R}^2$}\label{fig:circulant_matrices_2D}
\end{figure}

\begin{figure}[h!]
\centering
\begin{subfigure}[b]{0.45\textwidth}
\centering
\includegraphics[width=0.9\textwidth]{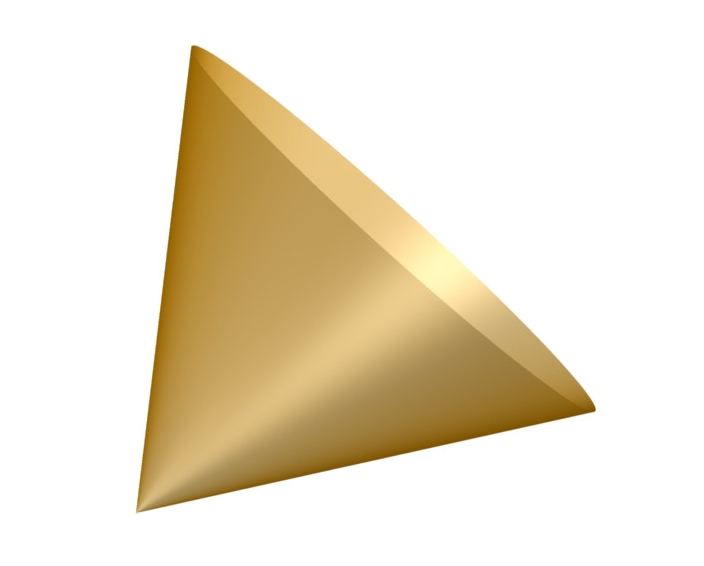}
\caption{Circulant matrices of psd rank at most 2}
\label{fig:circulantBoundary3D}
\end{subfigure}
\begin{subfigure}[b]{0.45\textwidth}
\centering
\includegraphics[width=0.9\textwidth]{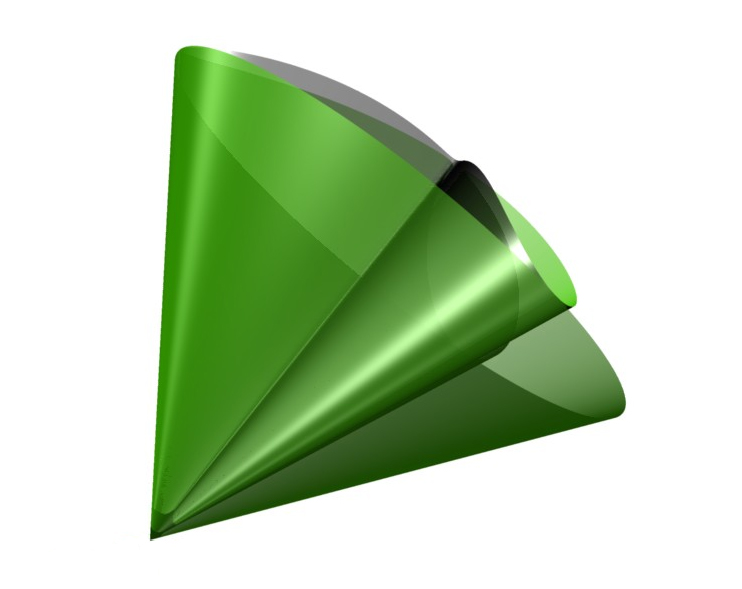}
\caption{The boundary polynomial}
\label{fig:regionsBoundaryPolynomial3D}
\end{subfigure}
\caption{$3 \times 3$ circulant matrices in $\mathbb{R}^3$}\label{fig:circulant_matrices_3D}
\end{figure}
\end{example}

\section{Matrices of higher psd rank}\label{sec:geometricInterpretation}

In Corollary~\ref{cor:algebraic}, we showed that a matrix lies on the boundary $\partial \mathcal M_{3,2}$ if and only if in every psd factorization $M_{ij} = \langle A_i, B_j\rangle$, at least three $A_i$'s and at least three $B_j$'s have rank one. In analogy with this result, we conjecture that a matrix lies on the boundary $\partial \mathcal M_{r,k}$ if and only if in every psd factorization $M_{ij} = \langle A_i, B_j\rangle$, at least $k+1$ matrices $A_i$ and at least $k+1$ matrices $B_j$ have rank one.

\begin{conjecture}\label{thm:k+1k}
A matrix $M\in\mathbb R^{p\times q}_{\geq 0}$ satisfying $M \mathbf 1=\mathbf1$ lies on the boundary $\partial \mathcal M_{r, k}$ if and only if for every size-$k$ psd factorization $M_{ij} = \langle A_i, B_j\rangle$, at least $k+1$ of the matrices $A_1,\dots, A_p\in\mathcal S^2_+$ have rank one and at least $k+1$ of the matrices $B_1,\dots, B_q\in\mathcal S^2_+$ have rank one.
\end{conjecture}

Let $M\in\mathbb R^{p \times q}_{\geq 0}$ be a full rank matrix, and let $P\subseteq Q\subseteq \mathbb R^{r-1}$ be nested polytopes such that $M = S_{P, Q}$. By Theorem~\ref{thm:spectrahedron}, the matrix $M$ has psd rank at most $k$ if and only if we can nest a spectrahedral shadow $C$ of size $k$ between $P$ and $Q$.
By definition, the spectrahedral shadow $C$ is a linear projection of a spectrahedron $\tilde C = L\cap\mathcal S^k_+$  of size $k$. 

\begin{definition}We say that a vector $v\in C$ lies in the {\em rank $s$ locus} of $C$ if there exists a $k\times k$ psd matrix in $\tilde C$ of rank $s$ that projects onto $v$.
\end{definition}

The geometric version of the Conjecture~\ref{thm:k+1k} is:

\begin{conjecture}\label{conjecture:geometric_description2}
A matrix $M$ is on the boundary $\partial \mathcal M_{r,k}$ if and only if all spectrahedral shadows $C$ of size $k$ such that $P \subseteq C \subseteq Q$ contain $k+1$  vertices of $P$ at rank one loci and touch $k+1$ facets of $Q$ at rank $k-1$ loci. 
\end{conjecture}

For $r=\binom{k+1}{2}$, one can show similarly to the proof of Corollary~\ref{cor:algebraic} that Conjectures~\ref{thm:k+1k} and~\ref{conjecture:geometric_description2} are equivalent. This case differs from other cases, by linear map $\pi$ being invertible.

The psd rank three and rank four setting corresponds to the geometric configuration where a $3$-dimensional spectrahedral shadow of size three is nested between $3$-dimensional polytopes. A detailed study of generic spectrahedral shadows can be found in~\cite{SS14}. 

\begin{example} \rm
We now give an example of a geometric configuration as in Conjecture~\ref{conjecture:geometric_description2}. We stipulate that the vertices of the interior polytope coincide with the nodes of the spectrahedron in Figure~\ref{figure:shadow1} and the facets of the outer polytope touch the boundary of this spectrahedron at rank two loci. In the dual picture, the vertices of the inner polytope lie on the rank one locus depicted in Figure~\ref{figure:shadow3} and the facets of the outer polytope contain the rank two locus of this spectrahedral shadow.

\begin{figure}[h]
\begin{subfigure}[b]{0.45\textwidth}
\centering
\includegraphics[width=0.85\textwidth]{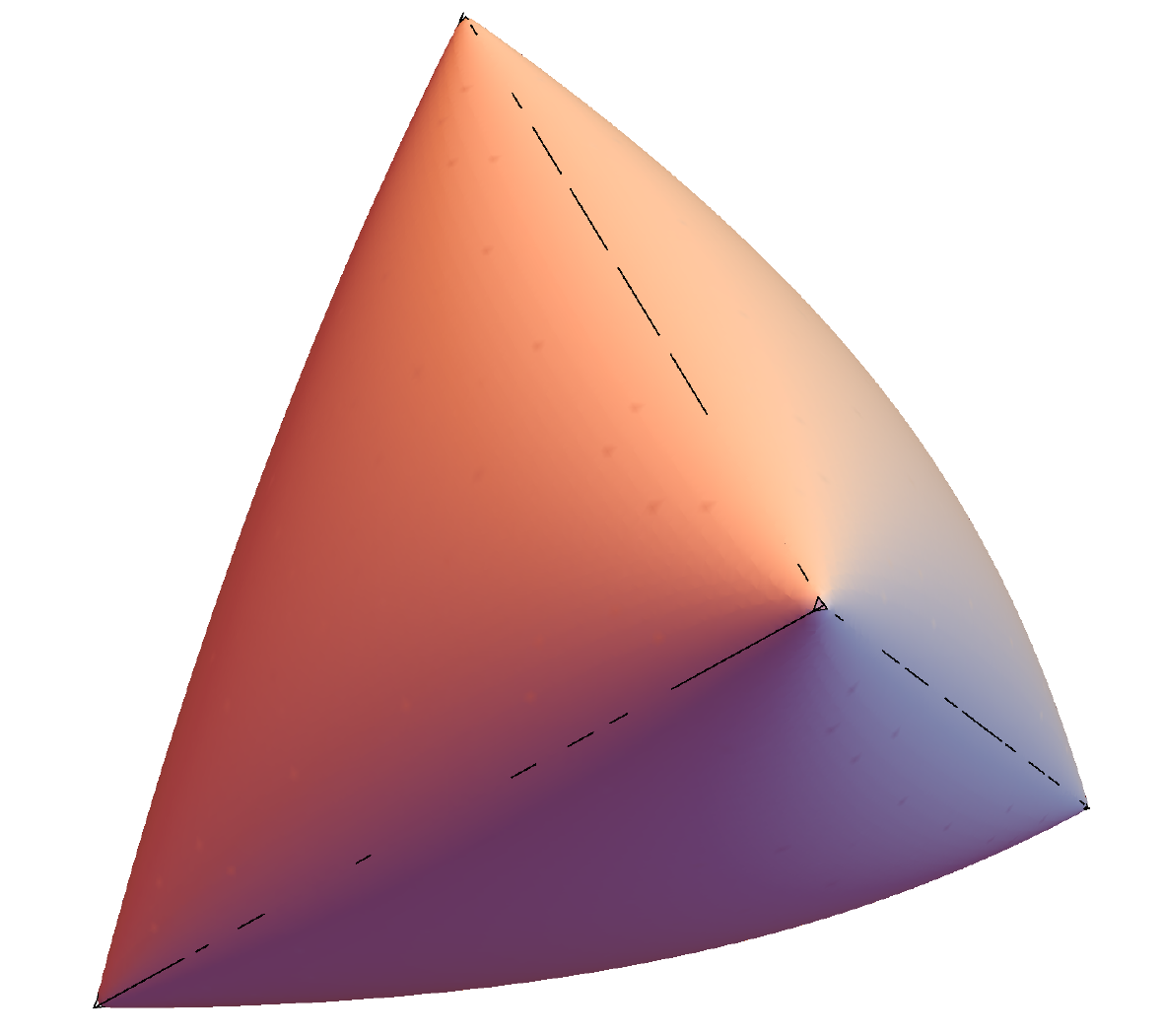}
\caption{Spectrahedron}
\label{figure:shadow1}
\end{subfigure}
\begin{subfigure}[b]{0.45\textwidth}
\centering
\includegraphics[width=0.9\textwidth]{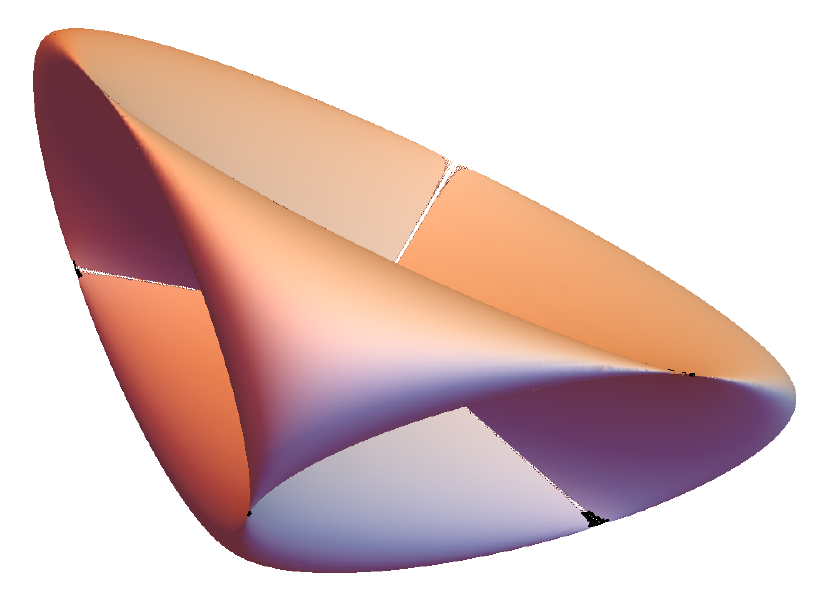}
\caption{Rank-one locus of the dual shadow}
\label{figure:shadow3}
\end{subfigure}
\caption{3-dimensional spectrahedral shadows}
\end{figure}
\end{example}

We end this section with a restatement of Conjecture~\ref{thm:k+1k} in a special case  using Hadamard square roots.

\begin{definition}
Given a nonnegative matrix $M$, let
$\sqrt{M}$ denote a Hadamard square root of $M$ obtained by replacing each entry in $M$ by one of its two possible square roots. The square root rank of a nonnegative matrix M, denoted as $\text{rank}_{\sqrt{}}(M)$, is the minimum rank of a Hadamard square root of $M$.
\end{definition}

\begin{lemma}[Lemma~2.4 in~\cite{GRT13}]
The smallest $k$ for which  a nonnegative real matrix  $M$ admits a $\mathcal{S}^k_+$-factorization in which all factors are matrices of rank one is $k=\text{rank}_{\sqrt{}}(M)$.
\end{lemma}

Hence Conjecture~\ref{thm:k+1k} is equivalent to the statement that a matrix $M\in\mathcal M_{k+1,k}^{(k+1) \times (k+1)}$ lies on the boundary $\partial \mathcal M_{k+1,k}^{(k+1) \times (k+1)}$ if and only if its square root rank is at most $k$.
We conclude this section with a conjecture which would lead to a semialgebraic description of $\mathcal M_{r, k}^{p \times q}$.

\begin{conjecture}\label{semialgebraic_description} 
Every matrix $M \in \mathcal{M}^{p \times q}_{r,k}$ has a psd factorization $M_{ij}=\langle A_i,B_j \rangle$, with at least $k$ matrices $A_i$ and $k-1$ matrices $B_j$, or at least $k-1$ matrices $A_i$ and $k$ matrices $B_j$ being rank one.
\end{conjecture}

If this conjecture were true, there would be $\binom{p}{k}\binom{q}{k-1}+\binom{p}{k-1}\binom{q}{k}$ options for selecting the $2k-1$ rank-one matrices. For each such option we would be able to describe the semialgebraic set of all such matrices that have psd rank $k$.

\section{Evidence towards Conjecture \ref{thm:k+1k}}\label{sec:higherPsdRank}

In this section, we present partial evidence towards proving Conjecture~\ref{thm:k+1k} if $p=q=r=k+1$. Section \ref{sec:5.1} is theoretical in nature, while Section \ref{section:higher_psd_rank} exhibits computational results.

\subsection{Nested spectrahedra} \label{sec:5.1}
By Theorem \ref{thm:spectrahedron} a matrix $M$ for which $M\mathbf1 = \mathbf1$ has psd rank $k$ if and only if we can nest a spectrahedral shadow of size $k$ between the polytopes $P$ and $Q$ corresponding to $M$. In the following lemma, we show that a $(k+1)\times (k+1)$ matrix $M$ has psd rank $k$ if and only if we can fit a spectrahedron of size $k$ between $P$ and $Q$. We show that if there is a spectrahedral shadow $C$ nested between $P$ and $Q$, then we can find a spectrahedron $C'$ of the same size such that $P\subseteq C'\subseteq C\subseteq Q$.

\begin{lemma}\label{lemma:enough_to_consider_spectrahedra}
Let $M\in\mathbb R^{(k+1)\times (k+1)}_{\geq 0}$ be a full-rank matrix such that $M\mathbf 1 = \mathbf 1$. Then, $M$ has psd rank at most $k$ if and only if we can nest a spectrahedron of size $k$ between the two polytopes $P$ and $Q$ corresponding to $M$.
\end{lemma}

\begin{proof}
 If we can fit a spectrahedron of size $k$ between $P$ and $Q$, then $M$ has psd rank at most $k$.

Conversely, suppose that $M$ has psd rank at most $k$. Then there exists a slice $L$ of $\mathcal S^k_+$ and a linear map $\pi$ such that $C = \pi(L\cap \mathcal S^k_+)$ lies between $P$ and $Q$:
$$P\subseteq C\subseteq Q.$$
If $\pi$ is a $1:1$ linear map, then, the image $C$ is just a linear transformation of a spectrahedron, and is therefore a spectrahedron of the same size. So, assume that $\pi$ is not $1:1$, i.e. it has nontrivial kernel.

We can write
$$L\cap \mathcal S^k_+ = \{(x_1,\dots, x_s) \in \R^s : \sum_{i=1}^s x_iA_i + (1-\sum_{i=1}^s x_i)A_{s+1}\succeq 0\}$$
for some $A_1,\dots, A_{s+1} \in \mathcal{S}^k$. Let $u_1,\dots, u_s$ be an orthonormal basis of $\mathbb R^s$ such that $\pi(u_i)=e_i$ for $i \in \{1,\ldots,r\}$ and $\text{ker}(\pi) = \text{span}(u_{k+1},\dots, u_s)$. Let $U$ be the orthogonal matrix with columns $u_1,\dots,u_s$. Consider new coordinates $y$ such that $x = Uy$. We can write
$$L\cap\mathcal S^k_+ = \{Uy \in \R^s : \sum_{i=1}^s y_i B_i + (1 - \sum_{i=1}^s y_i) B_{s+1}\succeq 0\},$$
where $B_1,\dots,B_{s+1}$ are linear combinations of the $A_i$'s. Then
$$C = \{(y_1,\dots,y_k) \in \R^k : \exists y_{k+1},\dots,y_s \in \R \text{ s.t.} \sum_{i=1}^s y_i B_i + (1 - \sum_{i=1}^s y_i)B_{s+1}\succeq 0\}.$$

 Since $M$ is full rank, we can factor it as $M = AB$, where $A, B\in\mathbb R^{(k+1)\times (k+1)}$ and
$$A = \begin{pmatrix} 1 & 0 & \cdots & 0 & 1\\
0 & 1 & \cdots & 0 & 1\\
\vdots & &\ddots &  & \vdots\\
0 & 0& \cdots & 1 & 1\\
0 & 0& \cdots & 0 & 1
\end{pmatrix}, \quad\quad\quad \quad\quad\quad  B = A^{-1} M.$$
The inner polytope $P$ comes from an affine slice of the conic hull of the rows of $A$. Let the slice be given by the last coordinate equal to 1. Then $P$ is the standard simplex in $\mathbb R^k$, i.e.
$$P = \text{conv}\{e_1,\dots,e_k, 0\}.$$

Since $e_i\in P\subseteq C$ for $i \in \{1,\ldots,k\}$, then there exist $y_{k+1}^{(i)},\dots, y_s^{(i)}\in\mathbb R$ such that
$$D_i = B_i + \sum_{j=k+1}^s [y_j^{(i)}(B_j - B_{s+1})]\succeq 0.$$
Since $0\in P\subseteq C$, then there exist $y_{k+1}^{(0)},\dots, y_s^{(0)}\in\mathbb R$ such that
$$D_{k+1} =B_{s+1} + \sum_{j=k+1}^s [y_j^{(0)}(B_j - B_{s+1})]\succeq 0.$$
Consider the spectrahedron
$$C' = \{(y_1,\dots,y_k) : \sum_{i=1}^k y_i D_i + (1-\sum_{i=1}^k y_i)D_{k+1} \succeq 0\}.$$
We have $e_i \in C'$ for $i \in \{1,\dots,k\}$,  since $D_i\succeq 0$. Also $0\in C'$, since $D_{k+1}\succeq 0$. Thus $P\subseteq C'$.

Moreover, if $(y_1,\dots,y_k)\in C'$, then
$$0 \preceq \sum_{i=1}^k y_i D_i + (1-\sum_{i=1}^k y_i)D_{k+1}= \sum_{i=1}^k y_i(B_i + \sum_{j=k+1}^s [y_j^{(i)}(B_j - B_{s+1})])$$
$$+ (1-\sum_{i=1}^k y_i)(B_{s+1} + \sum_{j=k+1}^s [y_j^{(0)}(B_j - B_{s+1})])$$
$$=\sum_{i=1}^k y_i B_i + \sum_{j=k+1}^s(\sum_{i=1}^ky_iy_j^{(i)} - (1-\sum_{i=1}^ky_i)y_j^{(0)}) B_j$$
$$+ (1 - \sum_{i=1}^ky_i - \sum_{j=k+1}^s(\sum_{i=1}^ky_iy_j^{(i)} -(1-\sum_{i=1}^ky_i)y_j^{(0)}))B_{s+1}.$$
Therefore $(y_1,\dots, y_k)\in C$ and $P\subseteq C'\subseteq C\subseteq Q$.
\end{proof}

We conjecture that the statement of Lemma~\ref{lemma:enough_to_consider_spectrahedra} holds for matrices of any size.

\begin{conjecture}\label{conjecture:spectrahedra_are_enough}
Let $M\in\mathbb R^{p\times q}_{\geq 0}$ have rank $k+1$ and assume that $M\mathbf 1 = \mathbf 1$. Then $M$ has psd rank at most $k$ if and only if we can nest a spectrahedron of size $k$ between the two polytopes $P$ and $Q$ corresponding to $M$.
\end{conjecture}

We now turn our attention to matrices which lie on the boundary of the set of matrices of fixed size, rank, and psd rank. Our goal is to present partial evidence towards Conjecture~\ref{conjecture:geometric_description2}. Suppose we have polytopes $P$ and $Q$ and a spectrahedron $C$ such that $P\subseteq C\subseteq Q$. Further, assume that $P$ has $k+1$ vertices. We show that if $k$ of the $k+1$ vertices of the polytope $P$ touch the spectrahedron $C$ at rank-one loci, then we can find a smaller spectrahedron $C'$ such that $P\subseteq C'\subseteq C\subseteq Q$. This means that the matrix $S_{P, Q}$ does not lie on the boundary $\partial \mathcal M_{k, k+1}^{(k+1)\times (k+1)}$.

\begin{lemma}\label{lemma:two_spectrahedra}
Let $P = \text{conv}(e_1,\dots, e_k, 0) \subseteq\mathbb R^k$. Let $C$ be a spectrahedron of size $k$ such that $P\subseteq C$ and the vertices $e_1,\dots, e_k$ correspond to rank one matrices in $C$. Then there exists another spectrahedron $C'$ of size $k$ such that $P\subseteq C'\subseteq C$ with all $k+1$ vertices of $P$ corresponding to rank one matrices in $C'$.
\end{lemma}

\begin{figure}[h]
\centering
\includegraphics[width=0.4\textwidth]{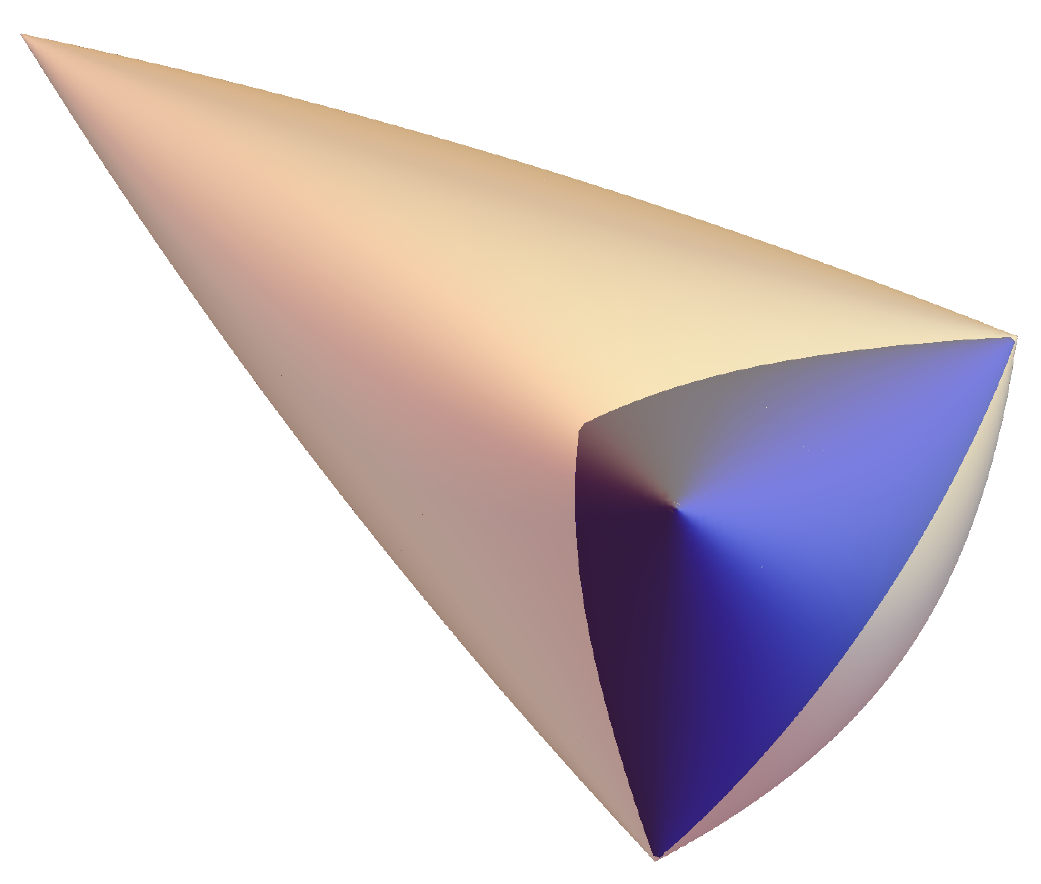}
\caption{The spectrahedra $C$ (in light yellow) and $C'$ (in blue) as in Lemma~\ref{lemma:two_spectrahedra}}
\label{fig:two_spectrahedra}
\end{figure}

\begin{proof}
The statement is trivial when $k=1$. We proceed by induction.

By the conditions in the statement of the lemma, we can assume that
$$C = \{(x_1,\dots, x_k) \in \R^k : x_1a_1a_1^T + x_2 a_2 a_2^T +\cdots + x_k a_k a_k^T + (1-\sum_{i=1}^k x_i) B \succeq 0\},$$
where  $a_1,\dots, a_k \in\mathbb R^k$ are vectors. We have $B\succeq 0$ since $0\in C$.

Suppose first that $\dim(\text{span}\{a_1,\dots, a_k\}) = \ell < k$. Let $U$ be a change of coordinates that transforms  span$\{a_1,\dots, a_k\}$ into span$\{e_1,\dots, e_l\}$. Denoting $a_i' = Ua_i$, we have
$$C = \{ (x_1,\dots, x_k) \in \R^k  : x_1a_1'(a'_1)^T + x_2 a_2' (a'_2)^T +\cdots + x_k a_k' (a'_k)^T + (1-\sum_{i=1}^k x_i) UBU^T \succeq 0\},$$
where $B' = UBU^T$ is positive semidefinite. If $B'_{i,j} = 0$ for all $i,j \geq \ell+1$, then, the statement reduces to the case of $\ell$, which is true by induction. So suppose that $B'_{\ell+1, \ell+1} > 0$  (since $B'\succeq 0$). Choose a vector $d\in\mathbb R^k$ such that $d_{\ell+1}\neq 0$ and $d d^T \preceq B'$. Consider the spectrahedron
$$C' = \{(x_1,\dots, x_k) \in \R^k : x_1a_1'(a'_1)^T + x_2 a_2' (a'_2)^T +\cdots + x_k a_k' (a'_k)^T + (1-\sum_{i=1}^k x_i) d d^T\succeq 0\}.$$
Clearly $e_1,\dots, e_k, 0\in C'$. We will show that $C'\subseteq C$. Indeed, let $(x_1,\dots, x_k)\in C'$. Since $(a_i')_{\ell+1} = 0$ for $i \in \{1, \ldots,k \}$, $d_{\ell+1} \neq 0$ and 
$$ x_1a_1'(a'_1)^T + x_2 a_2' (a'_2)^T +\cdots + x_k a_k' (a'_k)^T+ (1-\sum_{i=1}^k x_i) d d^T\succeq 0,$$
we have $(1-\sum_{i=1}^k x_i) \geq 0$. But then
$$ 0 \preceq x_1a_1'(a'_1)^T + x_2 a_2' (a'_2)^T +\cdots + x_k a_k' (a'_k)^T + (1-\sum_{i=1}^k x_i) d d^T $$
$$\preceq x_1a_1'(a'_1)^T + x_2 a_2' (a'_2)^T +\cdots + x_k a_k' (a'_k)^T + (1-\sum_{i=1}^k x_i) B'$$
and therefore $C'\subseteq C$.

Now assume that $\dim(\text{span}\{a_1,\dots, a_k\}) = k$. Let $U$ be an invertible transformation such that $Ua_i = e_i$. Then
$$C =\{(x_1,\dots, x_k) \in \R^k : x_1e_1e_1^T + x_2 e_2 e_2^T +\cdots + x_k e_k e_k^T + (1-\sum_{i=1}^k x_i) UBU^T \succeq 0\},$$
where $B' = UBU^T$ is positive semidefinite. Let $d\in\mathbb R^k$ be such that $d_i = \sqrt{B'_{i,i}}$ and let $S\in\mathbb R^{k\times k}$ be such that
$$S_{i,j} = \begin{cases} \frac{B'_{i,j}}{\sqrt{B'_{i,i}B'_{j,j}}} & \text{if } B'_{i,i}B'_{j,j} \neq 0,\\
1 & \text{if } B'_{i,i}B'_{j,j} = 0 \text{ and } i = j,\\
0 & \text{if } B'_{i,i}B'_{j,j} = 0 \text{ and } i\neq j.
\end{cases}$$
Since $B'\succeq 0$, also $S \succeq 0$, since it is obtained from $B'$ by rescaling some rows and columns and by adding $1$ on the diagonal in places that are 0 in $B'$.
Let
$$C' =\{(x_1,\dots, x_k) \in \R^k: x_1e_1e_1^T + x_2 e_2 e_2^T +\cdots + x_k e_k e_k^T + (1-\sum_{i=1}^kx_i) d d^T \succeq 0\}.$$
Then, clearly $e_1,\dots, e_k, 0\in C'$. We will show that $C'\subseteq C$. Let $(x_1,\dots, x_k)\in C'$. Then
\begin{align}\label{Cpoint}
x_1e_1e_1^T + x_2 e_2 e_2^T +\cdots + x_k e_k e_k^T + (1-\sum_{i=1}^k x_i) d d^T \succeq 0.
\end{align}
By the Schur Product Theorem, we know that the Hadamard product of two positive semidefinite matrices is positive semidefinite. Therefore, when we take the Hadamard product of the matrix \eqref{Cpoint} with $S$ we get a positive semidefinite matrix. But that Hadamard product equals
$$x_1 e_1e_1^T + x_2 e_2 e_2^T +\cdots + x_k e_k e_k^T + (1-\sum_{i=1}^k x_i) B' \succeq 0,$$
and therefore $C'\subseteq C$.

\end{proof}

Let $P$ and $C$ be as in the statement of Lemma~\ref{lemma:two_spectrahedra}. Let $Q \subset \R^k$ be any polytope such that $P \subseteq C \subseteq Q$ and  consider the slack matrix $S_{P, Q}$. The statement of Lemma~\ref{lemma:two_spectrahedra} indicates that $S_{P, Q}$ does not lie on the boundary $\partial\mathcal M^{(k+1)\times(k+1)}_{k+1, k}$, because the new spectrahedron $C'$ does not touch $Q$. As we saw in Section~\ref{section:rank3psdrank2}, in order for a matrix to lie on the boundary, the configuration $P\subseteq C\subseteq Q$ has to be very tight, and Lemma~\ref{lemma:two_spectrahedra} shows that having $k$ of the vertices of $P$ lie in the rank one locus of $C$ is not tight enough. Similarly, having $k$ of the facets of $Q$ touch $C$ at rank $k-1$ loci will not be enough. This is why we believe that all $k+1$ vertices of $P$ have to be in the rank one locus of $C$, and all $k+1$ of the facets of $Q$ have to touch $C$ at its rank $k-1$ locus.

\subsection{Computational evidence}\label{section:higher_psd_rank}

In this section we provide computational evidence for Conjecture~\ref{thm:k+1k} when $k>2$.

\begin{example} \rm
We consider the 2-dimensional family of $4 \times 4$ circulant matrices
\begin{equation}\label{circulantFamily}
\begin{bmatrix}
a & b & 1 & b\\
b & a & b & 1\\
1 & b & a & b\\
b & 1 & b & a
\end{bmatrix}
\end{equation}
which is parametrized by $a$ and $b$.

In Figure~\ref{fig:circulantPsd3}, the $4126$ green dots correspond to randomly chosen matrices of the form (\ref{circulantFamily}) that have psd rank at most three. The psd rank is computed using the code provided by the authors of~\cite{VGGT15} adapted to the computation of psd rank~\cite[Section 5.6]{KLTT15}. The red curves correspond to matrices of the form (\ref{circulantFamily}) that have a psd factorization by $3 \times 3$ rank one matrices. These curves are obtained by an elimination procedure in {\tt Macaulay2}.
\begin{figure}
\centering
\includegraphics[width=0.6\textwidth]{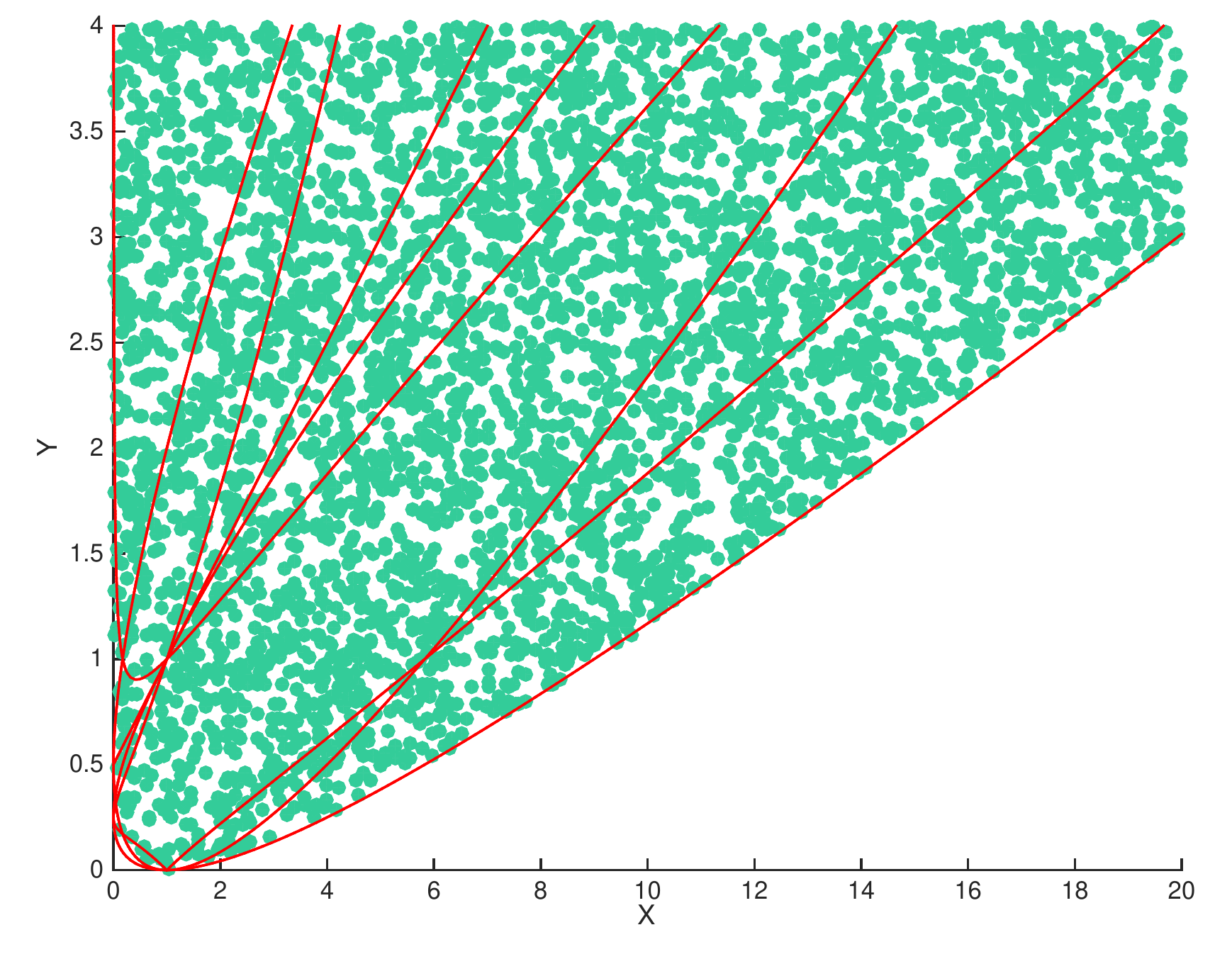}
\caption{A family of $4 \times 4$ circulant matrices of psd rank at most 3}
\label{fig:circulantPsd3}
\end{figure}
\end{example}

If the condition that $k+1$ matrices $A_i$ and $k+1$ matrices $B_j$ have rank one is equivalent to the matrix $M$ being on the algebraic boundary $\overline{\partial\mathcal M_{r, k}^{p\times q}}$, then the set of matrices that have a psd factorization by such matrices should have codimension one inside the variety $ \mathcal V_r^{p \times q}$ of $p \times q$ matrices of rank at most $r$. The dimension of $ \mathcal V_r^{p \times q}$ is $pr+qr-r^2$. 
In the following example, we test several different assignments of ranks to each of the matrices $A_i, B_j$, and we mark those whose image has dimension $pr+qr-r^2-1$.

\begin{example}\label{example:dimension_computation} \rm
Let  $A_1, \ldots , A_{p}, B_1, \ldots , B_{1} \in \mathcal{S}^k_+$ be symbolic matrices of ranks $r_1, \ldots, r_p,r'_1,\ldots,r'_q$. We construct a matrix $M$ such that $M_{ij}=\langle A_i, B_j \rangle$. We vectorize the matrix $M$ and compute its Jacobian $J$ with respect to the entries of $A_1, \ldots, A_{p}, B_1, \ldots, B_{q}$. Finally we substitute the entries of $A_1, \ldots, A_{p}, B_1, \ldots , B_{q}$ by random nonnegative integers and compute the rank of $J$ after this substitution. If $\rank(J)=pq-1$, then the matrices that have a psd factorization by matrices of ranks $\{r_1, \ldots , r_p\},\{r'_1, \ldots , r'_q\}$  give a candidate for a boundary component, assuming that the boundary components are only dependent on the ranks of the $A_i$'s and the $B_j$'s. 

\begin{table}[H]
\begin{tabular}{ | c | c | c | c | }
\hline
psd rank & p & q & ranks \\
\hline
3 & 4 & 4  & \{\{1,1,1,1\},\{1,1,1,1\}\}    \\
3 & 4 & 5 & \{\{1,1,1,1\},\{1,1,1,1,2/3\}\} \\
3 & 4 & 6 &  \{\{1,1,1,1\},\{1,1,1,1,2/3,2/3\}\},\{\{1,1,1,2\},\{1,1,1,1,1,1\}\} \\  
3 & 5 & 5 & \{\{1,1,1,1,2/3\},\{1,1,1,1,2/3\}\} \\
3 & 5 & 6 & \{\{1,1,1,1,2/3\},\{1,1,1,1,2/3,2/3\}\},\{\{1,1,1,2,3\},\{1,1,1,1,1,1\}\} \\
3 & 6 & 6 &\begin{tabular}{@{}c@{}}\{\{1,1,1,1,2/3,2/3\},\{1,1,1,1,2/3,2/3\}\},\{\{1,1,1,1,1,1\},\{1,1,1,2,3,3\}\},\\
\{\{1,1,1,1,1,1\},\{1,1,2,2,2,2\}\},\{\{1,1,1,1,1,2\},\{1,1,1,2,2,2\}\}\end{tabular} \\
\hline
\end{tabular}
\caption{Ranks of matrices in the psd factorization of a psd rank three matrix that can potentially give boundary components}
\label{table1}
\end{table}

The possible candidates for $k=3$ are summarized in Table~\ref{table1}. For all $p,q$ the case where four matrices $A_i$ and four matrices $B_j$ have rank one and all other matrices have any rank greater than one are represented. These are the cases that appear in Conjecture~\ref{thm:k+1k}. If any of  the other candidates in Table~\ref{table1} corresponded to a boundary component, then Conjecture~\ref{thm:k+1k} would be false.

If $k=4$, $p=q=10$, exactly five $A_i$ and five $B_j$ matrices have rank one and the rest of the matrices have rank two, then the Jacobian has rank 94. If the rest of the matrices in the psd factorization have rank three or four, then the Jacobian has rank 99 as expected. Hence if Conjecture~\ref{thm:k+1k} is true, then in general not every matrix on the boundary has a psd factorization with $k+1$ matrices $A_i$ and $k+1$ matrices $B_j$ having rank one, and rest of the matrices having rank two.
\end{example}

\begin{example} \rm
Using the same strategy as in Example~\ref{example:dimension_computation}, we have checked that the Jacobian has the expected rank for $p=q=r=k+1$ and $k<10$.
\end{example}

\appendix

\end{document}